\pdfoutput=1
\documentclass[11pt]{article}

\usepackage[a4paper,left=27mm,right=27mm,top=25mm,bottom=28mm]{geometry}
\usepackage{amsmath,amssymb,amsfonts,amsthm,mathtools,bm}
\usepackage{booktabs,array,enumitem,float,microtype,graphicx,subcaption}
\usepackage{xcolor}
\usepackage{hyperref}

\hypersetup{
  colorlinks=true,
  linkcolor=blue,
  citecolor=blue,
  urlcolor=blue,
  pdftitle={A Pressure-Free Virtual Element Method for the Surface Stokes Problem},
  pdfauthor={Jun Hu, Shengyang Xu, Hao Zhou},
  pdfkeywords={surface Stokes problem, virtual element method, nonconforming Stokes complex, pressure-robustness, surface topology, macroelement realization}
}
\allowdisplaybreaks

\newtheorem{theorem}{Theorem}[section]
\newtheorem{lemma}[theorem]{Lemma}

\theoremstyle{definition}

\newtheorem{assumption}[theorem]{Assumption}
\theoremstyle{remark}
\newtheorem{remark}[theorem]{Remark}
\numberwithin{equation}{section}

\newcommand{\R}{\mathbb R}
\newcommand{\Pk}{\mathbb P}
\newcommand{\Th}{\mathcal T_h}
\newcommand{\Eh}{\mathcal E_h}
\newcommand{\Vh}{\mathcal V_h}
\makeatletter
\newcommand{\norm}{\@ifnextchar[{\norm@with}{\norm@without}}
\def\norm@with[#1]#2{\lVert#2\rVert_{#1}}
\newcommand{\norm@without}[1]{\lVert#1\rVert}
\newcommand{\seminorm}{\@ifnextchar[{\seminorm@with}{\seminorm@without}}
\def\seminorm@with[#1]#2{\lvert#2\rvert_{#1}}
\newcommand{\seminorm@without}[1]{\lvert#1\rvert}
\newcommand{\widearc}{\mathpalette\svem@widearc}
\newcommand{\svem@widearc}[2]{\sbox\z@{$#1#2$}%
  \mathop{\vbox{\m@th\ialign{##\crcr
  \kern0.08em\svem@arcfill#1{0.8\wd\z@}\crcr
  \noalign{\nointerlineskip}$\hss#1#2\hss$\crcr}}}\nolimits}
\newcommand{\svem@arcfill}[2]{$\m@th\sbox\tw@{$#1($}%
  \hss\resizebox{#2}{\wd\tw@}{\rotatebox[origin=c]{-90}{\upshape(}}\hss$}
\makeatother
\newcommand{\bnu}{\bm\nu}
\newcommand{\bn}{\bm n}
\newcommand{\bt}{\bm t}
\newcommand{\bp}{\bm p}
\newcommand{\bx}{\bm x}
\newcommand{\bu}{\bm u}
\newcommand{\bv}{\bm v}
\newcommand{\bfv}{\bm f}
\newcommand{\bP}{\mathbf P}
\newcommand{\bH}{\mathbf H}
\newcommand{\bI}{\mathbf I}
\newcommand{\gradg}{\nabla_\gamma}
\newcommand{\gradh}{\nabla_h}
\newcommand{\divg}{\operatorname{div}_\gamma}
\newcommand{\divG}{\operatorname{div}_{\Gamma_h}}
\newcommand{\curlg}{\operatorname{curl}_\gamma}
\newcommand{\curlK}{\operatorname{curl}_K}
\newcommand{\divK}{\operatorname{div}_K}
\newcommand{\curlh}{\operatorname{curl}_h}
\newcommand{\divh}{\operatorname{div}_h}
\newcommand{\Phih}{\Phi_h^{r+1}}

\newcommand{\Sigmah}{\Sigma_h^r}
\newcommand{\Qh}{\mathring Q_h^{r-1}}

\newcommand{\jump}[1]{[\![#1]\!]}
\newcommand{\avg}[1]{\{#1\}}
\newcommand{\sym}{\operatorname{sym}}
\newcommand{\eps}{\varepsilon}

\title{A Pressure-Free Virtual Element Method for the Surface Stokes Problem}
\author{Jun Hu$^\dagger$, Shengyang Xu$^\dagger$, and Hao Zhou$^\dagger$}
\date{}

\begin{document}
\maketitle
\begingroup
\renewcommand{\thefootnote}{\fnsymbol{footnote}}
\footnotetext[2]{School of Mathematical Sciences, Peking University,
Beijing 100871, China (hujun@math.pku.edu.cn,
xushengyang0429@stu.pku.edu.cn, and zhouhao23@pku.edu.cn).}
\endgroup

\begin{abstract}
We develop a pressure-free virtual element method for the surface Stokes problem on polygonal approximations of closed surfaces of arbitrary genus. The method extends the planar VEM Stokes complex to discrete surfaces, yielding a nonconforming complex with an exactly tangential, \(H(\operatorname{div})\)-conforming velocity space, commuting interpolation, and the correct discrete cohomology. To obtain the weak consistency needed for optimal convergence, we introduce an edge correction that preserves the complex without adding degrees of freedom; the resulting construction reduces to the original complex in the planar setting. The virtual formulation requires no penalties, and a local divergence-preserving reconstruction provides a computable pressure-robust load. We also construct explicit macroelement representatives whose energy and load induce an alternative formulation on the VEM space, algebraically equivalent to a direct macroelement Galerkin method. The original virtual formulation and the macroelement method are stable and attain optimal first-order convergence in the broken \(H^1\) norm and second-order convergence in the \(L^2\) norm for the velocity. Numerical experiments illustrate these convergence rates, the role of the edge correction, and pressure robustness.
\end{abstract}

\noindent\textbf{Key words.} Surface Stokes problem, virtual element method,
Stokes complex, pressure-robustness, macroelement

\noindent\textbf{MSC codes.} 65N30, 65N12, 65N15, 76D07
\medskip

\section{Introduction}\label{sec:introduction}

The surface Stokes equations describe viscous incompressible flows on curved
surfaces, with applications to fluid membrane relaxation
\cite{arroyo2009relaxation}, particle mobility in curved membranes
\cite{henle2010hydrodynamics}, and the mechanics of viscous material
interfaces \cite{jankuhn2018incompressible}.
Let $\gamma\subset\mathbb R^3$ be a connected, closed, orientable surface
with tangential projector $\bP$. Given a tangential force $\bfv$, the problem
is to find a tangential velocity $\bu$ and a mean-zero pressure $p$ satisfying
\begin{equation}\label{eq:continuous-stokes}
 -\bP\divg\eps_\gamma(\bu)+\bu+\gradg p=\bfv,
 \qquad \divg\bu=0
 \quad\text{on }\gamma.
\end{equation}
Here $\gradg$ and $\divg$ denote the surface gradient and divergence, and
$\eps_\gamma(\bu)$ is the surface strain-rate tensor.
A discretization of \eqref{eq:continuous-stokes} must account for both the
tangential constraint and the incompressibility condition.

Surface finite element methods approximate~\eqref{eq:continuous-stokes}
in finite element spaces defined on a mesh approximation of $\gamma$.
On piecewise planar meshes, neighboring faces generally have different tangent
planes, making tangentiality and interelement continuity difficult to
impose together. One approach retains a continuous velocity space
and imposes tangentiality weakly through a normal-component penalty
\cite{hardering2025parametric}. Another approach enforces tangentiality
exactly. The surface BDM method controls interelement jumps by interior
penalties \cite{bonito2020divergence}, whereas surface MINI, Taylor--Hood,
and Scott--Vogelius methods couple degrees of freedom through
transformations without normal-component or edge-jump penalties
\cite{demlow2024tangential,demlow2026taylorhood,kone2026divergence}.

A further concern is pressure robustness: gradient forces that affect only
the pressure in the continuous problem should not contaminate the computed
velocity. This property may be lost when incompressibility is enforced only
against a discrete pressure space \cite{john2017divergence}.
On simply connected closed surfaces, stream-function formulations build
incompressibility into the velocity representation and eliminate pressure
\cite{reusken2020stream,brandner2020error}. Corresponding
pressure-free discretizations include a $C^0$ interior penalty method
\cite{neilan2025c0} and stabilized nonconforming methods
\cite{wu2025biharmonic,wuZhou2026morleyStokes}. On surfaces of positive genus,
however, a global stream function does not represent the full
divergence-free space. The additional harmonic components must also be
included, as in streamfunction--vorticity formulations and discrete
Helmholtz--Hodge decompositions
\cite{brueers2025streamfunction,brueers2026pressure}.

We develop a pressure-free virtual element method (VEM) by extending the
planar virtual element Stokes complex of~\cite{antonietti2014stream} to a
polygonal approximation $\Gamma_h$ of $\gamma$. A central issue in this
extension is to obtain the weak continuity needed for optimal velocity
approximation while preserving the complex and its topological structure.
The resulting nonconforming complex is
\begin{equation}\label{eq:intro-complex}
 0\longrightarrow\mathbb R\xrightarrow{\,\subset\,}\Phi_h
 \xrightarrow{\curlh}\Sigma_h
 \xrightarrow{\divh}\mathring Q_h\longrightarrow0.
\end{equation}
The velocity space $\Sigma_h$ is exactly tangential and
$H(\operatorname{div})$-conforming. The complex admits commuting
interpolation and preserves the cohomology dimension of the surface.

The surface assembly uses the nodal transformations
of~\cite{demlow2024tangential} to couple velocities across faces with
different tangent planes. For these virtual spaces, vertex coupling alone
does not provide the weak continuity required for optimal $L^2$
approximation. A vertex-based edge correction resolves this difficulty by
enforcing shared tangential edge averages while preserving the complex.
The correction adds no independent degrees of
freedom, and the construction reduces to the original lowest-order stream
virtual element complex in the planar setting.

The discrete Stokes problem is posed on the full divergence-free space
of~\eqref{eq:intro-complex}, yielding a pressure-free formulation valid for
arbitrary genus. The bilinear form is computable from the degrees of freedom
using only the polygonal mesh geometry and requires neither normal-component
nor edge-jump penalties. A local divergence-preserving reconstruction of
the virtual test functions makes the load computable for general $L^2$
forces. This reconstruction yields a pressure-robust load:
on a fixed polygonal surface $\Gamma_h$, adding a gradient to the force
leaves the discrete velocity unchanged.

Connections between virtual elements and explicit macro
spaces~\cite{chinosi2016virtual,chen2024extrinsic} suggest another route to
computability. The virtual degrees of freedom determine explicit
macroelement representatives. Their energy and load define an alternative
formulation on the virtual space. The induced formulation is
algebraically equivalent to the direct macroelement Galerkin method.
The representatives also preserve the boundary data and differential
structure, so this correspondence includes the full divergence-free spaces
and their topological components.

Both the original VEM and the macroelement method are stable and attain
optimal velocity convergence: first order in the broken $H^1$
norm and second order in the $L^2$ norm. For pure gradient forces
on $\gamma$, for which the exact velocity vanishes, the broken $H^1$ error
decays at second order. Numerical experiments on hybrid polygonal meshes
and surfaces of different genera confirm the convergence rates for both
methods, test the gradient-force estimate, and demonstrate the role of the
edge correction.

The paper is organised as follows.
Section~\ref{sec:geometry} introduces the geometry and local virtual element
spaces. Section~\ref{sec:global-complex} constructs the global complex and
commuting interpolants. Section~\ref{sec:method} develops the pressure-free
formulation and proves its stability.
Section~\ref{sec:explicit-representative} constructs the explicit
representatives and the associated macroelement method.
Section~\ref{sec:error-estimates} establishes error estimates for both
formulations, and Section~\ref{sec:numerical} presents the numerical
experiments. Section~\ref{sec:conclusions} discusses implications for
higher-order VEMs on curved surfaces and domains with curved boundaries.
Throughout, $A\lesssim B$ means that $A$ is bounded by $B$ up to a positive
constant, and $A\simeq B$ means $A\lesssim B$ and $B\lesssim A$.

\section{Preliminaries}
\label{sec:geometry}

We assume that $\gamma$ is of class $C^4$.
For sufficiently small $\delta>0$, let
$U_\delta:=\{x\in\R^3:\operatorname{dist}(x,\gamma)<\delta\}$ be a tubular
neighborhood of $\gamma$, and denote the signed distance by $d$.
Extend the unit normal and tangential projector to $U_\delta$ by
$\bnu:=\nabla d$ and $\bP:=\bI-\bnu\otimes\bnu$, where $\bI$ is the identity matrix.
The Weingarten map $\bH:=\nabla^2d$ and closest-point projection
$\bp(x):=x-d(x)\bnu(x)$ satisfy $\nabla\bp=\bP-d\bH$ in $U_\delta$.

\subsection{Surface geometry and differential operators}
\label{subsec:admissible-meshes}

For a face, edge, surface patch, or surface domain $D$, write
$(\cdot,\cdot)_D$ for the $L^2(D)$ inner product and
$\|\cdot\|_{m,D}$, $|\cdot|_{m,D}$ for the $H^m(D)$ norm and seminorm,
$m\ge0$.  These conventions apply componentwise, with the Frobenius product
for matrices.  For a scalar space $X(D)\subset L^2(D)$, set
$\mathring X(D):=\{w\in X(D):(w,1)_D=0\}$.  The symbol $I$ denotes the
identity operator on the space clear from context.

Set $\bm H^m(D):=[H^m(D)]^3$ and, for $D$ on an oriented surface with normal
$\bnu_D$, $\bm H_t^m(D):=\{\bv\in\bm H^m(D):\bv\cdot\bnu_D=0\}$ and
$\bm L_t^2(D):=\bm H_t^0(D)$.  For a field $w$ on $\gamma$, let
$w^e:=w\circ\bp$, componentwise for vectors and tensors.  For scalar $w$ and
 tangential $\bv$, define
$\gradg w:=\bP\nabla w^e$, $\curlg w:=\bnu\times\gradg w$,
$\nabla_\gamma\bv:=\bP\nabla\bv^e\bP$, and
$\divg\bv:=\operatorname{tr}(\nabla_\gamma\bv)$.  Set
$\eps_\gamma(\bv):=\sym(\nabla_\gamma\bv)
=\tfrac12(\nabla_\gamma\bv+(\nabla_\gamma\bv)^T)$.
A tensor $\bm A$ is tangential if $\bm A=\bP\bm A\bP$; its surface
 divergence is taken row-wise.  We write
$\bm H(\divg;\gamma):=\{\bv\in\bm L_t^2(\gamma):
\divg\bv\in L^2(\gamma)\}$.

Let $\Gamma_h\subset U_\delta$ be a connected, closed, coherently oriented
polyhedral approximation of $\gamma$ with a conforming partition $\Th$ into
planar polygons, edge set $\Eh$, and vertex set $\Vh$.  Set
$h_K:=\operatorname{diam}(K)$, $h_e:=|e|$, and
$h:=\max_{K\in\Th}h_K$, with $0<h\le h_0\le1$ and $h_0$ fixed sufficiently
small.  Denote the unit normal on $K$ by $\bnu_K$ and set
$\bP_K:=\bI-\bnu_K\otimes\bnu_K$,
$\bnu_h|_K:=\bnu_K$, and $\bP_h|_K:=\bP_K$.
For a vertex $a$ of $K$, set
$T_aK:=\{\bx\in\R^3:\bx\cdot\bnu_K=0\}$.
Let $\mathcal V(K)$ be the vertex set of $K$ and $\omega_K$ the union of
faces sharing a vertex with $K$; set $\omega_e:=K\cup L$ when
$e=\partial K\cap\partial L$.  For any union $\omega$ of mesh faces, write
$\omega^\gamma:=\bp(\omega)$.  When $\bp:\Gamma_h\to\gamma$ is bijective,
define $\mu_h$ by
$\mathrm d\sigma(\bp(x))=\mu_h(x)\,\mathrm d\sigma_h(x)$, where
$\mathrm d\sigma$ and $\mathrm d\sigma_h$ are the surface measures on
$\gamma$ and $\Gamma_h$.

\begin{assumption}[Admissible polygonal surface family]\label{ass:mesh}
There is a constant $\rho>0$, independent of $h$, such that:
\begin{enumerate}[label=\textnormal{(M\arabic*)},leftmargin=2.5em]
 \item \label{mesh:surface-geometry}
 The closest-point map $\bp:\Gamma_h\to\gamma$ is bijective, and every
 $K\in\Th$ satisfies
 \begin{equation}\label{eq:geometry-estimates}
  \norm[L^\infty(K)]{d}\lesssim h_K^2,
  \qquad \norm[L^\infty(K)]{\bnu-\bnu_K}\lesssim h_K.
 \end{equation}

 \item \label{mesh:star-shaped}
 Every $K\in\Th$ is star-shaped with respect to a disk of radius
 at least $\rho h_K$.

 \item \label{mesh:vertex-separation}
 Any two distinct vertices $a,b\in\mathcal V(K)$ satisfy
 $|a-b|\ge\rho h_K$.
\end{enumerate}
\end{assumption}

The geometry bounds also imply
$\norm[L^\infty(K)]{1-\mu_h}\lesssim h_K^2$.
Condition~\ref{mesh:surface-geometry} makes $\bp:\Gamma_h\to\gamma$ a
homeomorphism.  Conditions~\ref{mesh:star-shaped}--\ref{mesh:vertex-separation}
are standard in planar and surface VEMs
\cite[assumptions (A1)--(A2)]{beirao2019stokescomplex}
\cite[Section~4.1, (A1)--(A2)]{frittelli2018surfacevem}.  They give
$h_e\simeq h_K$ for $e\subset\partial K$, uniformly bounded numbers of
edges per face and faces per vertex patch, and uniformly comparable face
 diameters within each vertex patch.  For each $K$, fix a disk from
\ref{mesh:star-shaped} with center $\bx_K^c$.

On each face $K$, $\nabla_K$, $D_K^2$, and $\divK$ denote the planar
 gradient, Hessian, and divergence, with $|w|_{2,K}:=\|D_K^2w\|_{0,K}$.
Set $\Delta_Kw:=\divK\nabla_Kw$, $J_K\bv:=\bnu_K\times\bv$,
$\curlK w:=J_K\nabla_Kw$ and $\eps_K(\bv):=\sym(\nabla_K\bv)$
for tangential $\bv$.
The operators $\gradh$, $\curlh$, $\divh$, and $\eps_h$ act facewise
by their counterparts with subscript $K$.

For $m\ge0$, set
$H_h^m(\Gamma_h):=\{w\in L^2(\Gamma_h):w|_K\in H^m(K)\ \forall K\in\Th\}$,
$\bm H_{t,h}^m(\Gamma_h):=\{\bv\in[H_h^m(\Gamma_h)]^3:
\bv|_K\cdot\bnu_K=0\ \forall K\in\Th\}$, and
$\bm L_{t,h}^2(\Gamma_h):=\bm H_{t,h}^0(\Gamma_h)$.
For a face union $\omega$, set
\[
 \|w\|_{m,h,\omega}^2:=\sum_{K\subset\omega}\|w\|_{m,K}^2,
 \qquad
 |w|_{m,h,\omega}^2:=\sum_{K\subset\omega}|w|_{m,K}^2,
\]
and suppress $\omega$ when $\omega=\Gamma_h$.
For a face or edge $D$, let $\Pk_m(D)$ be the polynomials of total degree
at most $m$, with $\Pk_m(D)=\{0\}$ for $m<0$, and let $\Pi_m^D$ be the
$L^2(D)$ projector onto $\Pk_m(D)$.  Set
$\Pk_m(\Th):=\{q_h\in L^2(\Gamma_h):q_h|_K\in\Pk_m(K)\ \forall K\in\Th\}$
and $(\Pi_m w)|_K:=\Pi_m^K(w|_K)$; the projectors act componentwise on
vector and tensor fields.

\paragraph{Edge orientations and jumps.}

For each $e\in\Eh$, fix an ordered pair $(K,L)$ of adjacent faces with
$e=\partial K\cap\partial L$ and a unit tangent $\bt$ used on both sides.
For $e\subset\partial K$, let $\bn_K$ be the outward unit conormal in the
plane of $K$.  Write $\mathrm ds$ for arclength and
$\mathrm ds_\gamma$ when distinguishing the curved-edge measure.
For a scalar trace from $K$, set
$\partial_{\bt}w:=\bt\cdot\nabla_Kw$ and
$\partial_{\bn_K}w:=\bn_K\cdot\nabla_Kw$.  Define
$\sigma_{K,e}\in\{\pm1\}$ by
$J_K\bn_K=\sigma_{K,e}\bt$.  Then
$\bn_K=-\sigma_{K,e}J_K\bt$ and
$\sigma_{L,e}=-\sigma_{K,e}$.  Hence
\begin{equation}\label{eq:edge-curl-signs}
 \curlK\phi\cdot\bn_K=-\sigma_{K,e}\partial_{\bt}\phi,
 \qquad
 \curlK\phi\cdot\bt=\sigma_{K,e}\partial_{\bn_K}\phi.
\end{equation}
For scalar or ambient-vector traces, set 
\begin{equation}\label{eq:edge-jump-average}
 \avg{w}:=\tfrac12(w|_K+w|_L),
 \qquad
\jump{w}:=w|_K-w|_L.
\end{equation}
In particular, the tangential jump is
$\jump{\bv\cdot\bt}=(\bv|_K-\bv|_L)\cdot\bt$.
For outward-conormal quantities, write
\begin{equation}\label{eq:edge-conormal-jumps}
 \jump{\partial_{\bn}w}
 :=\partial_{\bn_K}w|_K+\partial_{\bn_L}w|_L,
 \quad
 \jump{\bv\cdot\bn}
 :=\bv|_K\cdot\bn_K+\bv|_L\cdot\bn_L,
 \quad
 \avg{\bv\cdot\bn}
 :=\tfrac12\bigl(\bv|_K\cdot\bn_K-\bv|_L\cdot\bn_L\bigr).
\end{equation}
With these conventions, \eqref{eq:edge-curl-signs} gives
$
 \jump{\curlh\phi\cdot\bt}
 =\sigma_{K,e}\jump{\partial_{\bn}\phi}.
$
We reserve $\divG$ for the distributional surface divergence and set
$\bm H(\divG;\Gamma_h):=\{\bv\in\bm L_{t,h}^2(\Gamma_h):
\divG\bv\in L^2(\Gamma_h)\}$.
A piecewise $H^1$ tangential field $\bv$ belongs to this space if and only if
$\jump{\bv\cdot\bn}=0$ on every edge.

\subsection{Surface Piola and vertex transforms}
\label{sec:surface-piola-transform}
\label{sec:vertex-transforms}

We use the contravariant Piola transform in
\cite[Section~2.2]{demlow2024tangential}, specialized to the closest-point map
\(\bp:\Gamma_h\to\gamma\) and the surface Jacobian \(\mu_h\) from
Subsection~\ref{subsec:admissible-meshes}.
For a tangential field \(\bv_h\) on \(K\), the forward Piola transform
to \(K^\gamma\) is given by
\[
 (\mathcal P_{\bp}\bv_h)^e
 =\mu_h^{-1}(\bP-d\bH)\bv_h.
\]
Conversely, for a tangential field \(\bv\) on \(K^\gamma\), its Piola
pullback is
\begin{equation}\label{eq:surface-piola}
 \breve\bv:=\mathcal P_{\bp^{-1}}\bv
 =\mu_h\left(\bI-
       \frac{\bnu\otimes\bnu_h}{\bnu\cdot\bnu_h}\right)
       (\bI-d\bH)^{-1}\bv^e.
\end{equation}
If $\bv$ is defined
on all of $\gamma$, then $\breve\bv$ is understood facewise on $\Gamma_h$, and
we write $\breve\bv_K:=(\breve\bv)|_K$ for its restriction to $K$.
For a tangential field $\bv_h$ on $\Gamma_h$, we write
$\widearc{\bv_h}:=\mathcal P_{\bp}\bv_h$.
The two transforms are inverse, so $\widearc{\breve\bv}=\bv$ and
$\breve{\widearc{\bv_h}}=\bv_h$.
The two divergence identities are
\begin{equation}\label{eq:surface-piola-div}
 \divh\breve\bv=\mu_h(\divg\bv)^e,
 \qquad
 (\divg\widearc{\bv_h})^e
 =\mu_h^{-1}\divh\bv_h.
\end{equation}
The normal trace is preserved as well.  More precisely, if
\(\bv\in\bm H_t^1(K^\gamma)\),
\(e\subset\partial K\), \(e^\gamma:=\bp(e)\), and
\(\bn_{K^\gamma}\) is the outward unit conormal to \(K^\gamma\), then
\begin{equation}\label{eq:surface-piola-normal-trace}
 \int_e \breve\bv\cdot\bn_K\,q^e\,\mathrm{d}s
 =\int_{e^\gamma}\bv\cdot\bn_{K^\gamma}\,q\,\mathrm{d}s
 \qquad\forall q\in L^2(e^\gamma).
\end{equation}
For a sufficiently smooth scalar function \(w\), it holds that
\begin{equation}\label{eq:surface-piola-curl}
 \mathcal P_{\bp^{-1}}(\curlg w)=\curlh w^e,
 \qquad
 \mathcal P_{\bp}(\curlh w^e)=\curlg w.
\end{equation}

The standard local estimates for the surface Piola map give, for $K\in\Th$,
$m\in\{0,1,2\}$, and $\bv\in\bm H_t^m(K^\gamma)$,
\begin{equation}\label{eq:surface-piola-norm-equivalence}
 \|\breve\bv\|_{m,K}
 \simeq
 \|\bv\|_{m,K^\gamma}.
\end{equation}
The reverse estimate holds for $\mathcal P_{\bp}$, with constants uniform
under Assumption~\ref{ass:mesh}; see
\cite[Lemma~2.1]{demlow2024tangential}.

For each vertex $a\in\Vh$, choose one reference face $K_a$.
For $K\ni a$, define the velocity vertex transform
\cite[Definition~2.3]{demlow2024tangential} by
\begin{equation}\label{eq:vertex-transform}
 M_{a,K}\bx
 :=(\bnu_{K_a}\cdot\bnu_K)\bx
   -\bnu_{K_a}(\bnu_K\cdot\bx).
\end{equation}
Given \(\bnu_{K_a}\cdot\bnu_K\) is uniformly positive, \(M_{a,K}:T_aK_a\to T_aK\) is an isomorphism.  It satisfies
\begin{equation}\label{eq:vertex-transform-commuting}
 J_K\bP_K\bm g=M_{a,K}J_{K_a}\bm g
 \qquad\forall\bm g\in T_aK_a.
\end{equation}

\begin{lemma}[Edge components of the vertex transform]
\label{lem:vertex-transform-edge}
Let $e=\partial K\cap\partial L$, let $a\in\partial e$, and let
\(\bm g\in T_aK_a\).  Then
\begin{equation}\label{eq:vertex-transform-flux}
 (M_{a,K}\bm g)\cdot\bn_K
 +(M_{a,L}\bm g)\cdot\bn_L=0,
\end{equation}
and
\begin{equation}\label{eq:vertex-transform-tangent}
 |(M_{a,K}\bm g-M_{a,L}\bm g)\cdot\bt|
 \lesssim h_e^2|\bm g|.
\end{equation}
The hidden constant is independent of $h$, $e$, and the vertex patch.
\end{lemma}

\begin{proof}
Identity~\eqref{eq:vertex-transform-flux} is the vertex identity in the
proof of \cite[Proposition~3.1]{demlow2024tangential}.
Moreover,
\[
 (M_{a,K}\bm g-M_{a,L}\bm g)\cdot\bt
 =\bnu_{K_a}\cdot(\bnu_K-\bnu_L)(\bm g\cdot\bt)
 -((\bnu_K-\bnu_L)\cdot\bm g)(\bnu_{K_a}\cdot\bt).
\]
For unit normals and $F\in\{K,L\}$,
$\bnu_{K_a}\cdot\bnu_F
=1-\tfrac12 |\bnu_{K_a}-\bnu_F|^2$.
The geometry bounds prove \eqref{eq:vertex-transform-tangent}.
\end{proof}

The relation between the Piola pullback and the vertex transform is
the following result \cite[Lemma~2.5]{demlow2024tangential}.
The tangential-jump estimate follows from the pointwise bound
$|\breve\bv_K-\bP_K\bv^e|\lesssim h_K^2|\bv^e|$ in
\cite[Lemma~2.3]{wu2025biharmonic} and $\bP_K\bt=\bt$.

\begin{lemma}[Vertex transform comparison]
\label{lem:vertex-transform-comparison}
Let $\bm g$ be tangential to $\gamma$ at $\bp(a)$, and let
$\breve{\bm g}_K(a)$ denote its Piola pullback to $K$ at $a$.  For every
face $K\ni a$,
\begin{equation}\label{eq:vertex-transform-error}
 |\breve{\bm g}_K(a)-M_{a,K}\breve{\bm g}_{K_a}(a)|
 \lesssim h_a^2|\breve{\bm g}_{K_a}(a)|
 \lesssim h_a^2|\bm g|,
 \qquad
 h_a:=\max_{K'\ni a}h_{K'}.
\end{equation}
\end{lemma}

\begin{lemma}[Tangential jump of Piola pullbacks]
\label{lem:edge-piola-tangential-comparison}
Let $e=\partial K\cap\partial L$ and
$\bv\in\bm H_t^1(\omega_e^\gamma)$. Then
\begin{equation}
 \|\jump{\breve\bv\cdot\bt}\|_{0,e}
 \lesssim h_e^2\left(
 h_e^{-1/2}\|\bv\|_{0,\omega_e^\gamma}
 +h_e^{1/2}|\bv|_{1,\omega_e^\gamma}\right).
 \label{eq:edge-piola-tangential-comparison}
\end{equation}
\end{lemma}

\subsection{The local virtual element spaces}\label{sec:local-complex}
We introduce the local spaces and degrees of freedom used to construct the
VEM Stokes complex.  In Section~\ref{sec:global-assembly}, the velocity
tangential edge moments and the stream-function edge degrees of freedom will be
determined by the vertex data, reducing the number of independent degrees
of freedom.
For $m\ge0$, write
\[
 B_m(\partial K):=\{w\in C^0(\partial K):w|_e\in\Pk_m(e)
       \ \forall e\subset\partial K\}.
\]

The local velocity space is
\begin{equation}\label{eq:local-velocity-space}
 \begin{aligned}
 \Sigma(K):=\{\bv\in[H^1(K)]^2:\;&
 \bv|_{\partial K}\in[B_2(\partial K)]^2,\quad
 \divK\bv\in\Pk_0(K),\\
 &-\Delta_K\bv-\nabla_Ks=0
       \text{ for some }s\in L^2(K)/\mathbb R\}.
 \end{aligned}
\end{equation}
In moment coordinates, its degrees of freedom are
\begin{enumerate}[label=\textnormal{(V\arabic*)},leftmargin=12mm]
 \item\label{dof:velocity-vertex}
       the vector value $\bv(a)$ at every vertex $a$ of $K$;
 \item\label{dof:velocity-edge}
       $\displaystyle\frac1{|e|}\int_e\bv\cdot\bn_K\,\mathrm{d}s$ and
       $\displaystyle\frac1{|e|}\int_e\bv\cdot\bt\,\mathrm{d}s$ for every $e\subset\partial K$.
\end{enumerate}

The local stream-function space is
\begin{equation}\label{eq:local-stream-space}
 \begin{aligned}
 \Phi(K):=\{\phi\in H^2(K):\;
 \phi|_{\partial K}\in B_3(\partial K),\quad
 \nabla_K\phi|_{\partial K}\in[B_2(\partial K)]^2,\quad
\Delta_K^2\phi=0\}.
 \end{aligned}
\end{equation}
For $\phi\in\Phi(K)$, take the following degrees of freedom:
\begin{enumerate}[label=\textnormal{(S\arabic*)},leftmargin=12mm]
 \item\label{dof:stream-vertex}
       $\phi(a)$ and $\nabla_K\phi(a)$ at every vertex $a$ of $K$;
 \item\label{dof:stream-edge-normal}
       $\displaystyle\frac1{|e|}\int_e\partial_{\bn_K}\phi\,\mathrm{d}s$ for every
       $e\subset\partial K$.
\end{enumerate}
The degrees of freedom \ref{dof:velocity-vertex}--\ref{dof:velocity-edge}
determine $\bv|_{\partial K}$.  The constant divergence follows from the normal
trace by the divergence theorem.  The degrees of freedom
\ref{dof:stream-vertex}--\ref{dof:stream-edge-normal} determine
$\phi|_{\partial K}$ and $\nabla_K\phi|_{\partial K}$.  Neither space has cell degrees of freedom.

The unisolvence of these two spaces follows from existence and uniqueness for the Stokes and biharmonic boundary-value
problems, while the polynomial inclusions follow directly
from the definitions.

\begin{lemma}[Local unisolvence and polynomial inclusion]
\label{lem:local-space-unisolvence}
Let $K\in\Th$ be a polygonal element.
\begin{enumerate}[label=\textnormal{(\roman*)},leftmargin=10mm]
 \item The degrees of freedom
       \ref{dof:velocity-vertex}--\ref{dof:velocity-edge} are unisolvent
       for $\Sigma(K)$, and $[\Pk_1(K)]^2\subset\Sigma(K)$.
 \item The degrees of freedom
       \ref{dof:stream-vertex}--\ref{dof:stream-edge-normal} are unisolvent
       for $\Phi(K)$, and $\Pk_3(K)\subset\Phi(K)$.
\end{enumerate}
\end{lemma}

On the simply connected polygon $K$, the divergence-free Stokes fields in
$\Sigma(K)$ are precisely the curls of the biharmonic functions in $\Phi(K)$.
In planar coordinates,
$\bx/2\in\Sigma(K)$ has divergence one, so the following local complex is
exact.
\begin{lemma}[Local VEM Stokes complex]
\label{lem:local-stokes-complex}
Let $K\in\Th$.  The local spaces defined in
\eqref{eq:local-velocity-space} and \eqref{eq:local-stream-space} form the
exact sequence
\[
 0\longrightarrow\mathbb R\xrightarrow{\,\subset\,}
 \Phi(K)\xrightarrow{\,\curlK\,}
 \Sigma(K)\xrightarrow{\,\divK\,}
 \Pk_0(K)\longrightarrow0.
\]
At the last two spaces, exactness reads
\[
 \divK\Sigma(K)=\Pk_0(K),
 \qquad \ker(\divK|_{\Sigma(K)})=\curlK\Phi(K).
\]
\end{lemma}

\begin{figure}[htbp]
 \centering
 \begin{subfigure}[t]{.235\linewidth}
  \centering\includegraphics[width=\linewidth]{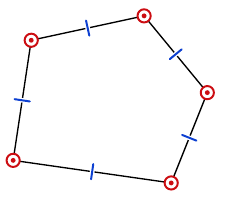}
  \caption{Virtual scalar}
 \end{subfigure}\hfill
 \begin{subfigure}[t]{.235\linewidth}
  \centering\includegraphics[width=\linewidth]{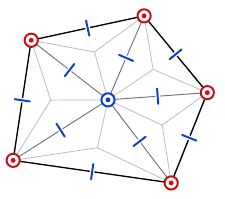}
  \caption{HCT scalar}
 \end{subfigure}\hfill
 \begin{subfigure}[t]{.235\linewidth}
  \centering\includegraphics[width=\linewidth]{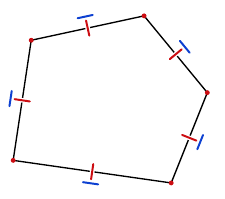}
  \caption{Virtual velocity}
 \end{subfigure}\hfill
 \begin{subfigure}[t]{.235\linewidth}
  \centering\includegraphics[width=\linewidth]{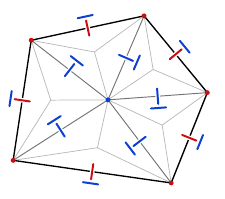}
  \caption{HCT velocity}
 \end{subfigure}
 \caption{Degrees of freedom of the local virtual element and macroelement constructed in Section~\ref{sec:explicit-representative}. Dots denote function values, circles denote scalar gradients, and strokes denote edge averages. Red degrees of freedom remain independent after assembly, whereas blue ones are constrained or eliminated locally.}
 \label{fig:paired-local-dofs}
\end{figure}

\section{A nonconforming Stokes complex on polygonal surfaces}
\label{sec:global-assembly}\label{sec:global-complex}

On the smooth surface $\gamma$, the Stokes complex is
\begin{equation}\label{eq:intro-surface-complex}
 0\longrightarrow\mathbb R\xrightarrow{\,\subset\,}H^2(\gamma)
 \xrightarrow{\curlg}\bm H_t^1(\gamma)
 \xrightarrow{\divg}\mathring L^2(\gamma)\longrightarrow0.
\end{equation}
It is exact on a simply connected surface: the divergence is onto the
zero-mean pressure space, and every divergence-free $H^1$ velocity is the
curl of an $H^2$ scalar function, unique up to a constant.
On a surface of positive genus, harmonic tangential fields represent the
cohomology at the velocity space
\cite{holst2012geometric,reusken2020stream}.
In this section, we
construct a polygonal-surface counterpart with the expected cohomology
dimension and commuting interpolants.

\subsection{Global spaces and the discrete complex}
\renewcommand{\Phih}{\Phi_h}
\renewcommand{\Sigmah}{\Sigma_h}
\renewcommand{\Qh}{\mathring Q_h}
\paragraph{Global spaces.}
We assemble the local complex of Lemma~\ref{lem:local-stokes-complex}
using the vertex transformations, normal-flux continuity, and edge integral
constraints determined by the reference-face vertex data.

The pressure space is $\Qh:=\mathring{\Pk}_0(\Th)$.

For a piecewise vector field $\bv_h$, write $\bv_K:=\bv_h|_K$ and set
\begin{equation}\label{eq:shared-edge-tangential-mean}
 T_e(\bv_h):=\frac12\sum_{a\in\partial e}
       \bv_{K_a}(a)\cdot\bt.
\end{equation}
This value depends on the two reference-face vertex values and the edge
tangent, and is the same for both faces incident to $e$.

For a piecewise vector field $\bv_h$ with $\bv_K\in\Sigma(K)$, impose
\begin{subequations}\label{eq:global-velocity-assembly}
\begin{alignat}{2}
 \bv_K(a)&=M_{a,K}\bv_{K_a}(a)
 &&\quad\forall K\ni a,
 \label{eq:global-velocity-vertex-assembly}\\*
 \int_e\jump{\bv_h\cdot\bn}\,\mathrm{d}s&=0
 &&\quad\forall e\in\Eh,
 \label{eq:global-velocity-normal-assembly}\\*
 \frac{1}{|e|}\int_e\bv_K\cdot\bt\,\mathrm{d}s&=T_e(\bv_h)
 &&\quad\forall K\in\Th,\ e\subset\partial K.
 \label{eq:global-velocity-tangent-assembly}
\end{alignat}
\end{subequations}
These conditions define
\begin{equation}\label{eq:global-velocity-space}
 \Sigmah:=\{\bv_h:\bv_h|_K\in\Sigma(K)\ \forall K\in\Th,
              \ \eqref{eq:global-velocity-assembly}
              \text{ holds}\}.
\end{equation}
Its independent degrees of freedom are \ref{dof:velocity-vertex} on the
reference faces $K_a$ and the normal-flux moments in
\ref{dof:velocity-edge}.  The tangential moments in
\ref{dof:velocity-edge} are determined by the vertex data through
\eqref{eq:global-velocity-tangent-assembly}.

For a piecewise function $\phi_h$ with
$\phi_K:=\phi_h|_K\in\Phi(K)$, impose
\begin{subequations}\label{eq:global-stream-assembly}
\begin{alignat}{2}
 \phi_K(a)&=\phi_{K_a}(a)
 &&\quad\forall K\ni a,
 \label{eq:global-stream-value-assembly}\\
 \nabla_K\phi_K(a)&=\bP_K\nabla_{K_a}\phi_{K_a}(a)
 &&\quad\forall K\ni a,
 \label{eq:global-stream-vertex-assembly}\\
\frac{1}{|e|} \int_e\partial_{\bn_K}\phi_K\,\mathrm{d}s
 &=\sigma_{K,e}T_e(\curlh\phi_h)
 &&\quad\forall K\in\Th,\ e\subset\partial K.
 \label{eq:global-stream-normal-assembly}
\end{alignat}
\end{subequations}
These conditions define the auxiliary scalar space
\begin{equation}\label{eq:global-stream-space}
 \Phih:=\{\phi_h:\phi_h|_K\in\Phi(K)\ \forall K\in\Th,
             \ \eqref{eq:global-stream-assembly}
             \text{ holds}\}.
\end{equation}
Its independent degrees of freedom are \ref{dof:stream-vertex}, taken on
the reference faces $K_a$. 
On each edge $e$, the quadratic trace
$\partial_{\bn_K}\phi_K|_e\in\Pk_2(e)$ is determined by its endpoint
values and the integral prescribed by
\eqref{eq:global-stream-normal-assembly}.  On a planar mesh, this constraint forces $\partial_{\bn_K}\phi_K|_e\in\Pk_1(e)$.

\paragraph{Conformity.}
The global tangential $\bm H^1$ and scalar $H^2$ conformity of the smooth-surface
Stokes complex \eqref{eq:intro-surface-complex} cannot be carried over
directly to $\Gamma_h$.
However, the vertex transformations and normal-flux conditions preserve
conformity in the
$H^1(\Gamma_h)$--$\bm H(\divG;\Gamma_h)$--$L^2(\Gamma_h)$ de Rham domains.
In addition, the tangential components of the velocities and the
conormal derivatives of the scalar functions satisfy higher-order jump estimates,
providing approximate continuity in the remaining trace components.

\begin{lemma}[Conformity of the global spaces]
\label{lem:global-domain-conformity}
\begin{equation}\label{eq:global-domain-conformity}
 \Sigmah\subset\bm H(\divG;\Gamma_h),
 \qquad
 \Phih\subset H^1(\Gamma_h).
\end{equation}
Moreover, for every $\bv_h\in\Sigmah$, $\phi_h\in\Phih$, and $e\in\Eh$,
\[
 \int_e\jump{\bv_h\cdot\bt}\,\mathrm{d}s=0,
 \qquad
 \int_e\jump{\partial_{\bn}\phi_h}\,\mathrm{d}s=0.
\]
\end{lemma}

\begin{proof}
For the velocity space, \eqref{eq:global-velocity-vertex-assembly},
\eqref{eq:global-velocity-normal-assembly}, and
Lemma~\ref{lem:vertex-transform-edge} make the two outward conormal traces
cancel across each edge.  Hence
$\Sigmah\subset\bm H(\divG;\Gamma_h)$.
The vertex conditions \eqref{eq:global-stream-value-assembly}--\eqref{eq:global-stream-vertex-assembly} identify
the endpoint values and tangential derivatives of the cubic scalar traces,
since $\bP_K\bt=\bP_L\bt=\bt$ for $e=\partial K\cap\partial L$.  Thus
$\Phih\subset H^1(\Gamma_h)$.
The zero-mean jump conditions follow directly from
\eqref{eq:global-velocity-tangent-assembly} and
\eqref{eq:global-stream-normal-assembly}, since
$\sigma_{L,e}=-\sigma_{K,e}$.
\end{proof}

\begin{lemma}[Higher-order trace estimates]
\label{lem:higher-order-trace-estimates}
Let $e=\partial K\cap\partial L$.  For every $\bv_h\in\Sigmah$ and
$\phi_h\in\Phih$, it holds that
\begin{subequations}\label{eq:higher-order-trace-estimates}
\begin{align}
 \|\jump{\bv_h \cdot\bt}\|_{0,e}
 +h_e\|\partial_{\bt}(\jump{\bv_h \cdot\bt})\|_{0,e}
 &\lesssim h_e^{3/2}
 \left(\|\bv_h\|_{0,\omega_e}
       +h_e|\bv_h|_{1,h,\omega_e}\right),
 \label{eq:velocity-tangential-jump-lifting}\\
 \|\jump{\partial_{\bn}\phi_h}\|_{0,e}
 +h_e\|\partial_{\bt}\jump{\partial_{\bn}\phi_h}\|_{0,e}
 &\lesssim h_e^{3/2}
 \left(|\phi_h|_{1,h,\omega_e}
       +h_e|\phi_h|_{2,h,\omega_e}\right).
 \label{eq:normal-jump-lifting}
\end{align}
\end{subequations}
The hidden constants depend only on the constants in Assumption~\ref{ass:mesh}.
\end{lemma}

\begin{proof}
By the local trace definition and Lemma~\ref{lem:global-domain-conformity},
$\jump{\bv_h\cdot\bt}$ belongs to $\Pk_2(e)$ with zero mean.
Its endpoint values determine it, so
\[
 \|\jump{\bv_h\cdot\bt}\|_{0,e}
 +h_e\|\partial_{\bt}\jump{\bv_h\cdot\bt}\|_{0,e}
 \lesssim h_e^{1/2}\sum_{a\in\partial e}
       |\jump{\bv_h\cdot\bt}(a)|.
\]
For each endpoint $a$, \eqref{eq:global-velocity-vertex-assembly} and
Lemma~\ref{lem:vertex-transform-edge} give
\[
\begin{aligned}
 |\jump{\bv_h\cdot\bt}(a)|
 &\lesssim h_e^2|\bv_{K_a}(a)|
 \lesssim h_e^2|\bv_K(a)|\\
 &\lesssim h_e\left(\|\bv_K\|_{0,K}+h_e|\bv_K|_{1,K}\right).
\end{aligned}
\]
Here the second inequality follows from
\eqref{eq:global-velocity-vertex-assembly} and the uniformly bounded inverse
of $M_{a,K}:T_aK_a\to T_aK$; the last follows from the endpoint inverse estimate,
the scaled trace inequality, and $h_K\simeq h_e$.
Combining the estimates proves \eqref{eq:velocity-tangential-jump-lifting}.
The scalar estimate \eqref{eq:normal-jump-lifting} follows analogously using
\eqref{eq:vertex-transform-commuting} and \eqref{eq:edge-curl-signs}.
\end{proof}

\paragraph{Edge correction.}
A simpler construction starts from the lowest-order local velocity space of
\cite[Section~3.1]{antonietti2014stream},
\begin{equation}\label{eq:planar-lowest-order-velocity-space}
 \widetilde\Sigma(K):=\{\bv\in\Sigma(K):
       \bv\cdot\bt|_e\in\Pk_1(e)
       \ \forall e\subset\partial K\}.
\end{equation}
Its independent degrees of freedom are \ref{dof:velocity-vertex} and the
normal-flux moments in \ref{dof:velocity-edge}; vertex and normal-flux
assembly gives
\begin{equation}\label{eq:uncorrected-global-velocity-space}
 \widetilde\Sigma_h:=\{\bv_h:\bv_h|_K\in\widetilde\Sigma(K)
 \ \forall K\in\Th,\quad
 \eqref{eq:global-velocity-vertex-assembly}\text{--}
\eqref{eq:global-velocity-normal-assembly}\text{ hold}\}.
\end{equation}
Our choice $\Sigmah$ can be viewed as an edge correction of
$\widetilde\Sigma_h$. We first enrich the linear tangential edge traces to
quadratic traces in $\Sigma(K)$, and then reduce the enriched space by
prescribing their means through \eqref{eq:global-velocity-tangent-assembly}.
For fixed vertex and normal-flux data, each corrected tangential trace
differs from the linear trace by a multiple of the quadratic edge bubble
that vanishes at both endpoints. Its coefficient is determined by the
shared mean $T_e$, while the normal trace is unchanged. Thus the correction
enforces zero-mean tangential jumps without adding independent degrees of
freedom.

Both global spaces lie in $\bm H(\divG;\Gamma_h)$. The linear tangential jumps in
$\widetilde\Sigma_h$ satisfy \eqref{eq:velocity-tangential-jump-lifting}
by the same endpoint estimate, but need not have the zero-mean property of
Lemma~\ref{lem:global-domain-conformity}.  The spaces coincide on planar
meshes; see Remark~\ref{rem:planar-reduction}.

The correction is introduced to recover optimal $L^2$ convergence.
The shared edge means provide the cancellation needed for second-order
weak consistency, as explained in Remark~\ref{rem:uncorrected-surface-consistency}
of Section~\ref{sec:error-estimates}. The comparison in
Table~\ref{tab:svem-edge-correction} of Section~\ref{sec:numerical} shows
second-order $L^2$ convergence for $\Sigmah$ and first-order convergence
for $\widetilde\Sigma_h$.

\begin{remark}[Planar reduction]\label{rem:planar-reduction}
Consider a polygonal mesh of a planar domain $\Omega$, with a common unit
normal and boundary conditions imposed separately. Then $M_{a,K}$ and
$\bP_K$ reduce to the identity on the common tangent plane. For an edge
$e=[a,b]$, the quadratic tangential trace
$q=(\bv_K\cdot\bt)|e$ satisfies
$
\frac1{|e|}\int_e q \,\mathrm ds=\frac12 (q(a)+q(b)),
$
and hence $q\in\Pk_1(e)$. Therefore the local traces, degrees of freedom, and assembly rules reduce exactly to those of the planar construction in
\cite[Sections~3.1 and~4]{antonietti2014stream}; in particular,
$\Sigma_h=\widetilde\Sigma_h$. The standard planar conformity is thus
recovered, with $\Phi_h\subset H^2(\Omega)$ and
$\Sigma_h\subset[H^1(\Omega)]^2$.
\end{remark}

\paragraph{The complex.}
To describe exactness at the velocity space, define the discrete divergence
kernel
\begin{equation}\label{eq:global-discrete-kernel}
 \mathcal Z_h:=\{\bv_h\in\Sigmah:\divh\bv_h=0\}.
\end{equation}

\begin{theorem}[Exactness and cohomology]
\label{thm:global-exactness-cohomology}
Let $\Gamma_h$ be connected, closed, and orientable.  Then
\begin{equation}\label{eq:global-discrete-complex}
 0\longrightarrow\mathbb R\xrightarrow{\,\subset\,}
 \Phih\xrightarrow{\,\curlh\,}
 \Sigmah\xrightarrow{\,\divh\,}
 \Qh\longrightarrow0
\end{equation}
is a complex and is exact except possibly at $\Sigmah$.  The cohomology at
$\Sigmah$ is
\begin{equation}\label{eq:algebraic-cohomology-quotient}
 \mathcal C_h^1:=\mathcal Z_h\big/\curlh\Phih,
 \qquad
 \dim\mathcal C_h^1=2-\chi(\Gamma_h),
 \qquad
 \chi(\Gamma_h):=N_V-N_E+N_F,
\end{equation}
where $N_V:=\#\Vh$, $N_E:=\#\Eh$, and $N_F:=\#\Th$.
In particular, if $\Gamma_h$ is simply connected, then
$\chi(\Gamma_h)=2$ and \eqref{eq:global-discrete-complex} is exact.
\end{theorem}

\begin{proof}
We first consider the divergence.  $\divK\bv_K\in\Pk_0(K)$ for $\bv_h\in\Sigmah$, and conformity gives
$\int_{\Gamma_h}\divh\bv_h=0$.  Thus $\divh\Sigmah\subset\Qh$.
To prove surjectivity of $\divh$, let $q_h\in\Qh$, set
$m_K=(q_h,1)_K$ and $\bm m=(m_K)_{K\in\Th}$, and let
$B\in\mathbb R^{N_F\times N_E}$ be an incidence matrix of the connected dual
graph.  Since $\bm m\perp\bm1$, there is
$\bm F=(F_e)$ such that $B\bm F=\bm m$.  For each edge, choose
$\eta_e\in\Pk_2(e)$ that vanishes at its endpoints and satisfies
$\int_e\eta_e\,\mathrm{d}s=1$.  On each face $K$, prescribe zero vertex data and
\[
 \bv_K\cdot\bn_K=B_{Ke}F_e\eta_e
 \quad\text{and}\quad \bv_K\cdot\bt=0
 \quad\text{on }e\subset\partial K.
\]
The zero vertex data give $T_e=0$, so the prescribed tangential traces
satisfy \eqref{eq:global-velocity-tangent-assembly}.
Local unisolvence gives $\bv_K\in\Sigma(K)$, while the divergence theorem
determines its constant divergence:
\[
 \int_K\divK\bv_K
 =\sum_{e\subset\partial K}B_{Ke}F_e
 =(B\bm F)_K=m_K=\int_Kq_h.
\]
Hence $\divK\bv_K=q_h|_K$.  The incidence signs assemble the normal data,
while the prescribed vertex and tangential data agree across each edge.
Therefore the local fields form $\bv_h\in\Sigmah$ with
$\divh\bv_h=q_h$.

We next consider the curl.  Local exactness gives
$(\curlh\phi_h)|_K\in\Sigma(K)$.  The vertex data assemble by
\eqref{eq:vertex-transform-commuting}, and the scalar traces give
normal-flux continuity.  The scalar constraint
\eqref{eq:global-stream-normal-assembly} and the edge identities
\eqref{eq:edge-curl-signs} give
\[
\frac{1}{|e|} \int_e(\curlK\phi_K)\cdot\bt\,\mathrm{d}s
 =T_e(\curlh\phi_h),
\]
which is \eqref{eq:global-velocity-tangent-assembly}.
Hence $\curlh\Phih\subset\Sigmah$, and the facewise
identity $\divK\curlK=0$ gives $\curlh\Phih\subset\mathcal Z_h$.
If $\curlh\phi_h=0$, then $\phi_h$ is constant on every face; scalar trace
conformity and connectedness yield
$\ker(\curlh|_{\Phih})=\mathbb R$.

It remains to compute the cohomology at $\Sigmah$.
Exactness at $\Qh$ and $\Phih$, together with the assembled
degree-of-freedom counts, gives
\[
\begin{aligned}
 \dim\mathcal C_h^1
 &=\dim\Sigmah-\dim\Qh-\dim\Phih+1\\
 &=(2N_V+N_E)-(N_F-1)-3N_V+1
 =-N_V+N_E-N_F+2.
\end{aligned}
\]
\end{proof}

\subsection{Commuting interpolation and approximation estimates}
We construct commuting interpolants for the velocity and scalar spaces and
establish their approximation properties.
In the local data below, $K\in\Th$, $a\in\mathcal V(K)$, and
$e\subset\partial K$.
For $q\in \mathring L^2(\gamma)$, define the interpolant $I_Qq\in\Qh$ by
\begin{equation}\label{eq:divergence-data-interpolation}
(I_Qq)|_K:=\Pi_0^K(\mu_h q^e),
\end{equation}
where $\Pi_0^K$ denotes the $L^2(K)$-orthogonal projection onto the constants. By the definition of $\mu_h$, $I_Qq$ has zero integral over $\Gamma_h$.

For a tangential field $\bv\in\bm H_t^2(\gamma)$, denote its Piola pullback to
$K$ by $\breve\bv_K$. The velocity interpolant
$I_\Sigma\bv\in\Sigmah$ is specified by
\begin{subequations}\label{eq:velocity-interpolation-data}
\begin{align}
 (I_\Sigma\bv)|_K(a)&:=M_{a,K}\breve\bv_{K_a}(a),
 \label{eq:general-velocity-interpolant-vertex}\\
 \int_e(I_\Sigma\bv)|_K\cdot\bn_K\,\mathrm{d}s
 &:=\int_e\breve\bv_K\cdot\bn_K\,\mathrm{d}s.
 \label{eq:interpolation-edge-flux-data}
\end{align}
\end{subequations}

For $w\in H^3(\gamma)$, the scalar interpolant
$I_\Phi w\in\Phih$ is specified by
\begin{equation}\label{eq:stream-interpolation-data}
\begin{aligned}
 (I_\Phi w)|_K(a)&=w^e(a),\\
 \nabla_K(I_\Phi w)|_K(a)&=\bP_K\nabla_{K_a}w^e(a).
\end{aligned}
\end{equation}

The remaining edge integrals follow from the global constraints.
Local unisolvence and the assembly rules show that these data define global
interpolants $I_\Sigma\bv$ and $I_\Phi w$.  The assembly maps use only the mesh normals.

\begin{lemma}[Commuting interpolation]
\label{lem:commuting-interpolation}
The interpolants $I_\Phi$, $I_\Sigma$, and $I_Q$ make the following
diagram commute:
\begin{equation}\label{eq:commuting-interpolation-diagram}
\begin{array}{cccccccccc}
 0\longrightarrow&\mathbb R&\xrightarrow{\ \subset\ }&
 H^3(\gamma)&\xrightarrow{\quad\curlg\quad}&
 \bm H_t^2(\gamma)&\xrightarrow{\quad\divg\quad}&
 \mathring L^2(\gamma)&\longrightarrow&0\\[2pt]
&{\scriptstyle\mathrm{id}}\big\downarrow&&{\scriptstyle I_\Phi}\big\downarrow&&
 {\scriptstyle I_\Sigma}\big\downarrow&&
 {\scriptstyle I_Q}\big\downarrow \\[2pt]
 0\longrightarrow&\mathbb R&\xrightarrow{\ \subset\ }&
 \Phih&\xrightarrow{\quad\curlh\quad}&
 \Sigmah&\xrightarrow{\quad\divh\quad}&\Qh&\longrightarrow&0.
\end{array}
\end{equation}
In particular, $I_\Sigma\bv\in\mathcal Z_h$ whenever $\divg\bv=0$.
\end{lemma}

\begin{proof}
For the divergence square, \eqref{eq:surface-piola-div} and the divergence
theorem give
\begin{align*}
 (\divK(I_\Sigma\bv)|_K,q)_K=(\divK\breve\bv_K,q)_K
 =\bigl(\mu_h(\divg\bv)^e,q\bigr)_K
 \qquad\forall q\in\Pk_0(K).
\end{align*}
Since the discrete divergence is constant, this proves
$\divh I_\Sigma\bv=I_Q(\divg\bv)$.

For the curl square, \eqref{eq:surface-piola-curl} gives
$
 \mathcal P_{\bp^{-1}}(\curlg w)=\curlh w^e.
$
The normal fluxes agree by the endpoint values and
\eqref{eq:edge-curl-signs}, while $\divK\curlK=0$.  At each
$a\in\mathcal V(K)$,
\[
 \curlK(I_\Phi w)|_K(a)
 =J_K\bP_K\nabla_{K_a}w^e(a)
 =M_{a,K}J_{K_a}\nabla_{K_a}w^e(a)
 =\bigl(I_\Sigma(\curlg w)\bigr)|_K(a),
\]
where the middle equality is \eqref{eq:vertex-transform-commuting}.
The vertex identity also makes the reconstructed means $T_e$ agree.
Local unisolvence now proves $\curlh I_\Phi w=I_\Sigma(\curlg w)$.
Finally, \eqref{eq:stream-interpolation-data} and local unisolvence give
$I_\Phi c=c$ for every $c\in\mathbb R$.
\end{proof}

We first bound local velocities by their boundary traces.

\begin{lemma}[Local velocity estimates]
\label{lem:local-velocity-estimates}
Let $K$ be a face in a mesh family satisfying Assumption~\ref{ass:mesh}.
For every $\bv\in\Sigma(K)$,
\begin{equation}
 h_K^{-1}\|\bv\|_{0,K}+|\bv|_{1,K}
 \lesssim h_K^{-1/2}\|\bv\|_{0,\partial K}.
 \label{eq:local-velocity-dof-estimate}
\end{equation}
The hidden constants depend only on the constants in
Assumption~\ref{ass:mesh}.
\end{lemma}
\begin{proof}
The lower stability bound of \cite[Theorem~2.1]{meng2023stokesstability},
applied with polynomial degree two, gives
\[
 |\bv|_{1,K}
 \lesssim \|\divK\bv\|_{0,K}+h_K^{-1/2}\|\bv\|_{0,\partial K}.
\]
Since $\divK\bv$ is constant, the divergence theorem yields
$\|\divK\bv\|_{0,K}\lesssim h_K^{-1/2}\|\bv\|_{0,\partial K}$.
The boundary-mean Poincar\'e inequality then proves
\eqref{eq:local-velocity-dof-estimate}.
\end{proof}

\begin{lemma}[Approximation of the velocity interpolant]
\label{prop:general-velocity-interpolation}
For every $\bv\in\bm H_t^2(\gamma)$ and every $K\in\Th$,
\begin{equation}
 \|\breve\bv_K-(I_\Sigma\bv)|_K\|_{0,K}
 +h_K|\breve\bv_K-(I_\Sigma\bv)|_K|_{1,K}
 \lesssim h_K^2\|\bv\|_{2,\omega_K^\gamma}.
 \label{eq:general-velocity-interpolation-estimate}
\end{equation}
\end{lemma}

\begin{proof}
Choose an averaged Taylor polynomial $\boldsymbol p_1\in[\Pk_1(K)]^2$
of $\breve\bv_K$.  Scaled approximation and the $H^2$ embedding give
\begin{equation}\label{eq:velocity-affine-approximation}
 \|\breve\bv_K-\boldsymbol p_1\|_{0,K}
 +h_K|\breve\bv_K-\boldsymbol p_1|_{1,K}
 +h_K\|\breve\bv_K-\boldsymbol p_1\|_{L^\infty(K)}
 \lesssim h_K^2|\breve\bv_K|_{2,K}.
\end{equation}

Set $\boldsymbol\delta_K:=(I_\Sigma\bv)|_K-\boldsymbol p_1\in\Sigma(K)$.
Its vertex values are controlled by \eqref{eq:velocity-affine-approximation} and
\eqref{eq:vertex-transform-error}.  By
\eqref{eq:interpolation-edge-flux-data}, its normal edge integrals equal
those of $\breve\bv_K-\boldsymbol p_1$.
For $e\subset\partial K$, the tangential assembly rule and the fact
that $\boldsymbol p_1$ is affine give
\begin{equation}\label{eq:velocity-delta-tangential-moment}
\begin{aligned}
 \left|\frac1{|e|}\int_e\boldsymbol\delta_K\cdot\bt\,\mathrm{d}s\right|
 =\frac12\left|\sum_{a\in\partial e}
       (\breve\bv_{K_a}(a)-\boldsymbol p_1(a))\cdot\bt\right|
 \lesssim h_K|\breve\bv_K|_{2,K}
       +\sum_{a\in\partial e}h_a^2|\bv(\bp(a))|,
\end{aligned}
\end{equation}
The last bound uses \eqref{eq:vertex-transform-error} and \eqref{eq:velocity-affine-approximation}.

The quadratic boundary trace of $\boldsymbol\delta_K$ is determined by
these vertex values and edge integrals.  Polynomial norm equivalence on
the edges and Lemma~\ref{lem:local-velocity-estimates} therefore imply
\begin{equation}\label{eq:velocity-delta-bound}
\begin{aligned}
 h_K^{-1}\|\boldsymbol\delta_K\|_{0,K}
 +|\boldsymbol\delta_K|_{1,K}
 &\lesssim h_K^{-1/2}\|\boldsymbol\delta_K\|_{0,\partial K}\\
 &\lesssim h_K|\breve\bv_K|_{2,K}
       +\sum_{a\in\mathcal V(K)}h_a^2|\bv(\bp(a))|\\
 &\lesssim h_K\|\bv\|_{2,\omega_K^\gamma}.
\end{aligned}
\end{equation}
In the second inequality, the vertex and normal-moment terms use
\eqref{eq:velocity-affine-approximation}, \eqref{eq:vertex-transform-error},
and the interpolation data; the tangential-moment term uses
\eqref{eq:velocity-delta-tangential-moment}.
The last inequality uses the scaled $H^2$ embedding on $K_a^\gamma$,
$K_a\subset\omega_K$, $h_a\simeq h_K$, and
\eqref{eq:surface-piola-norm-equivalence}.
Combining \eqref{eq:velocity-affine-approximation} and
\eqref{eq:velocity-delta-bound} with the triangle inequality proves
\eqref{eq:general-velocity-interpolation-estimate}.
\end{proof}

Taking $\bv=\curlg w$ in Lemma~\ref{prop:general-velocity-interpolation}
and using the commuting relation yields the following scalar estimates.

\begin{lemma}[Approximation of the scalar interpolant]
\label{prop:general-paired-interpolation}
Let $\Gamma_h$ and $\Th$ satisfy
Assumption~\ref{ass:mesh}.  If $w\in H^3(\gamma)$, then, for every
$K\in\Th$,
\begin{subequations}\label{eq:paired-interpolation-estimates}
\begin{alignat}{2}
 \|\mathcal P_{\bp^{-1}}(\curlg w)-\curlK I_\Phi w\|_{m,K}
 &\lesssim h_K^{2-m}\|w\|_{3,\omega_K^\gamma}
 &\qquad m&=0,1,
 \label{eq:velocity-interpolation-target}\\
 \|w^e-I_\Phi w\|_{m,K}
 &\lesssim h_K^{3-m}\|w\|_{3,\omega_K^\gamma}
 &\qquad m&=0,1,2.
 \label{eq:scalar-interpolation-target}
\end{alignat}
\end{subequations}
\end{lemma}

\section{The virtual element method}
\label{sec:method}

Let
\[
 \mathcal Z(\gamma):=
 \{\bv\in\bm H_t^1(\gamma):\divg\bv=0\}.
\]
The pressure-free formulation of \eqref{eq:continuous-stokes} is to find
$\bu\in\mathcal Z(\gamma)$ such that
\begin{equation}\label{eq:continuous-kernel-stokes}
 a_\gamma(\bu,\bv)
 :=(\eps_\gamma(\bu),\eps_\gamma(\bv))_\gamma+(\bu,\bv)_\gamma
 =(\bfv,\bv)_\gamma
 \qquad\forall\bv\in\mathcal Z(\gamma).
\end{equation}
We discretize \eqref{eq:continuous-kernel-stokes} on the full discrete
divergence kernel $\mathcal Z_h$ defined in
\eqref{eq:global-discrete-kernel}.

\subsection{The discrete formulation}
\label{subsec:global-discrete-problem}
\label{subsec:local-stream-forms}

Since the local divergence is constant, a function in the space
\[
 \mathcal Z(K):=\{\bv\in\Sigma(K):\divK\bv=0\}
 =\{\bv\in\Sigma(K):
     \sum_{e\subset\partial K}\int_e\bv\cdot\bn_K\,\mathrm ds=0\}
\]
is uniquely determined by all local velocity degrees of freedom except
one normal-flux moment.
The elementwise Stokes energy is
\begin{equation}
 a_K(\bv,\boldsymbol z)
 :=(\eps_K\bv,\eps_K\boldsymbol z)_K+(\bv,\boldsymbol z)_K.
 \label{eq:continuous-face-stokes-form}
\end{equation}

We use polynomial projections and a stabilization term to obtain a
computable discrete form.
For $\bv,\boldsymbol z\in\mathcal Z(K)$, define the discrete local Stokes form by
\begin{equation}
\begin{aligned}
 a_{h,K}^V(\bv,\boldsymbol z)
 :=a_K\bigl(\Pi_1^{\nabla,K}\bv,
                  \Pi_1^{\nabla,K}\boldsymbol z\bigr)
 +S_K^\Sigma
 \bigl((I-\Pi_1^{\nabla,K})\bv,
       (I-\Pi_1^{\nabla,K})\boldsymbol z\bigr).
\end{aligned}
\label{eq:discrete-local-stokes-form}
\end{equation}
The $H^1$-seminorm projector
$\Pi_1^{\nabla,K}:[H^1(K)]^2\to[\Pk_1(K)]^2$ is defined by
\begin{subequations}
\label{eq:hessian-projector-method}
\begin{align}
 (\nabla_K\Pi_1^{\nabla,K}\bv,\nabla_K\boldsymbol q)_K
 &=(\nabla_K\bv,\nabla_K\boldsymbol q)_K
 \qquad\forall\boldsymbol q\in[\Pk_1(K)]^2,
 \label{eq:hessian-projector-energy}\\
 \int_K\Pi_1^{\nabla,K}\bv&=\int_K\bv.
 \label{eq:hessian-projector-kernel}
\end{align}
\end{subequations}
These conditions uniquely determine the projector and give, for
$\bv\in[H^1(K)]^2$,
\begin{equation}\label{eq:projected-stream-velocity-means}
 \eps_K\Pi_1^{\nabla,K}\bv=\Pi_0^K(\eps_K\bv),
 \qquad
 \divK\Pi_1^{\nabla,K}\bv=\Pi_0^K\divK\bv.
\end{equation}
So the projector preserves
$\mathcal Z(K)$.

To define the stabilization, order the vertices $a_i$ along the positive
boundary orientation and let $e_i=[a_i,a_{i+1}]$, with cyclic indices.
For $\bv\in\mathcal Z(K)$, define the cumulative flux coordinates
$\xi_{K,i}(\bv)$ by
\[
 \xi_{K,i+1}(\bv)-\xi_{K,i}(\bv)
 =-\int_{e_i}\bv\cdot\bn_K\,\mathrm ds,
 \qquad \sum_i\xi_{K,i}(\bv)=0.
\]
The zero total flux makes these coordinates well defined.  The stabilization
$S_K^\Sigma$ is the inner product of the scaled velocity data: for
$\bv,\boldsymbol z\in\mathcal Z(K)$,
\[
\begin{aligned}
 \boldsymbol d_{\Sigma,K}(\bv)&:=
 \left((h_K^{-1}\xi_{K,i}(\bv))_i,\ (\bv(a))_a\right),\\
 S_K^\Sigma(\bv,\boldsymbol z)&:=
 \boldsymbol d_{\Sigma,K}(\bv)\cdot\boldsymbol d_{\Sigma,K}(\boldsymbol z).
\end{aligned}
\]
The stabilization uses only the independent vertex and normal-flux data of
the assembled space; the tangential moments are prescribed by
\eqref{eq:global-velocity-tangent-assembly}.

We assemble the local forms in this space.  For the load, let
$\mathcal R_h^{\rm c}$ be the velocity reconstruction of Subsection~\ref{subsec:local-companion}.
For $\bfv_h\in\bm L_{t,h}^2(\Gamma_h)$ and
$\bv_h,\boldsymbol z_h\in\mathcal Z_h$, set
\begin{subequations}\label{eq:divergence-free-induced-forms}
\begin{align}
 a_h^Z(\bv_h,\boldsymbol z_h)
 &:={}
 \sum_{K\in\Th}
 a_{h,K}^V\bigl(\bv_h|_K,\boldsymbol z_h|_K\bigr),
 \label{eq:divergence-free-bilinear-form}\\*
 \ell_h^Z(\bfv_h;\boldsymbol z_h)
 &:={}
 (\bfv_h,\mathcal R_h^{\rm c}\boldsymbol z_h)_{\Gamma_h}.
 \label{eq:divergence-free-load}
\end{align}
\end{subequations}
The discrete problem seeks $\bu_h\in\mathcal Z_h$ such that
\begin{equation}
 a_h^Z(\bu_h,\boldsymbol z_h)
 =\ell_h^Z(\bfv_h;\boldsymbol z_h)
 \qquad\forall\boldsymbol z_h\in\mathcal Z_h.
 \label{eq:topology-completed-stokes-method}
\end{equation}

The computability of the discrete bilinear form follows from the following lemma.
\begin{lemma}[Computability]
\label{lem:companion-data-computability}
For every $\bv\in\mathcal Z(K)$, the polynomial $\Pi_1^{\nabla,K}\bv$ is computable
from the degrees of freedom \ref{dof:velocity-vertex}--\ref{dof:velocity-edge}.
\end{lemma}

\begin{proof}
The degrees of freedom determine $\bv|_{\partial K}$.
Integration by parts expresses the gradient integrals in
\eqref{eq:hessian-projector-method} in terms of this trace.
Since $\divK\bv=0$, the identity
$\int_K v_i=\int_{\partial K}x_i\bv\cdot\bn_K\,\mathrm ds$
determines the mean in local planar coordinates.  These data determine
$\Pi_1^{\nabla,K}\bv$.
\end{proof}

\paragraph{Pressure-free implementation.}
To enforce the divergence-free constraint without a pressure multiplier,
we decompose $\mathcal Z_h$ into the discrete curl space and a
finite-dimensional complement.
Set $\Phi_{h,0}:=\{\phi_h\in\Phih:
\sum_{a\in\Vh}\phi_h(a)=0\}$ and choose an algebraic complement
$\mathcal H_h\subset\mathcal Z_h$ such that
$\mathcal Z_h=\curlh\Phi_{h,0}\oplus\mathcal H_h$.
Theorem~\ref{thm:global-exactness-cohomology} gives
$\dim\mathcal H_h=2-\chi(\Gamma_h)$.
Using this cellular description, choose oriented edge-flux vectors with
zero signed sum on each face whose classes modulo vertex differences form
a cohomology basis. Set their vertex-velocity data to zero and use local
unisolvence to obtain fields in $\mathcal Z_h$; their span is a possible
choice of $\mathcal H_h$.
Writing $\bu_h=\curlh\phi_h+\boldsymbol\eta_h$, problem
\eqref{eq:topology-completed-stokes-method} is equivalent to finding
$(\phi_h,\boldsymbol\eta_h)\in\Phi_{h,0}\times\mathcal H_h$ such that
\begin{equation}\label{eq:equivalent-stream-formulation}
\begin{aligned}
 a_h^Z(\curlh\phi_h+\boldsymbol\eta_h,
       \curlh\psi_h+\boldsymbol\zeta_h)
 =\ell_h^Z(\bfv_h;\curlh\psi_h+\boldsymbol\zeta_h)
 \qquad\forall(\psi_h,\boldsymbol\zeta_h)\in\Phi_{h,0}\times\mathcal H_h.
\end{aligned}
\end{equation}
The system is assembled from the velocity degrees of freedom of the curl
basis functions and a basis of $\mathcal H_h$, using the same forms $a_h^Z$ and
$\ell_h^Z$.

\subsection{Pressure-robust load discretization}
\label{subsec:local-companion}
For a virtual velocity $\boldsymbol z$, the natural local load
$(\bfv_h,\boldsymbol z)_K$ is not directly computable.  A directly computable choice
is $(\bfv_h,\Pi_1^{\nabla,K}\boldsymbol z)_K$.  However, the projected velocities need
not form a globally $\bm H(\divG;\Gamma_h)$-conforming field, so this load need not annihilate gradient forces.  We therefore
construct an explicit local velocity reconstruction $\mathcal R_K$ such that, for
$\boldsymbol z_h\in\mathcal Z_h$, the reconstruction
\begin{equation}
 (\mathcal R_h^{\rm c}\boldsymbol z_h)|_K
 :=\mathcal R_K(\boldsymbol z_h|_K)
 \label{eq:divergence-kernel-companion-reconstruction}
\end{equation}
is globally divergence-free.  The load \eqref{eq:divergence-free-load} is
then computable from $L^2$ data and annihilates gradient forces.
We use the local constrained minimization approach of
\cite[Section~4.2]{frerichs2022pressure}, with a piecewise quadratic
$\bm H(\operatorname{div};K)$ space and the full normal trace prescribed.

Join $\bx_K^c$ to the vertices of $K$ and denote the resulting fan by
$\widehat{\mathcal T}_K$.  The choice made after
Assumption~\ref{ass:mesh} and the vertex-separation condition imply that the
fan is uniformly shape regular and has a uniformly bounded number of
triangles.  For $\bv\in\mathcal Z(K)$, set
\[
\begin{aligned}
 \boldsymbol W_K(\bv):=\{\boldsymbol r\in\bm H(\operatorname{div};K):\;&
 \boldsymbol r|_T\in[\Pk_2(T)]^2
          \ \forall T\in\widehat{\mathcal T}_K,\quad
 \operatorname{div}\boldsymbol r=0,\\
 &\boldsymbol r\cdot\bn_K=\bv\cdot\bn_K
          \ \text{on }\partial K\}.
\end{aligned}
\]
For $\bv\in\mathcal Z(K)$, let
$\mathcal R_K\bv\in\boldsymbol W_K(\bv)$ be the solution of the local minimization
problem
\begin{equation}
 \mathcal R_K\bv
 :=\underset{\boldsymbol r\in\boldsymbol W_K(\bv)}{\operatorname{argmin}}
       \|\boldsymbol r-\Pi_1^{\nabla,K}\bv\|_{0,K}^2.
\label{eq:poisson-companion}
\end{equation}
The $\mathrm{BDM}_2$ interpolant of $\bv$ on $\widehat{\mathcal T}_K$
belongs to $\boldsymbol W_K(\bv)$, since it preserves the quadratic boundary
normal trace and has zero divergence.  The strictly convex minimization
problem is therefore uniquely solvable.
For $\bv\in\mathcal Z(K)$, the degrees of freedom determine
$\bv\cdot\bn_K|_{\partial K}$ and $\Pi_1^{\nabla,K}\bv$.
Thus $\mathcal R_K\bv$ is computable.

\begin{lemma}[Properties of the local velocity reconstruction]
\label{prop:scalar-companion}
The map $\mathcal R_K$ is linear, preserves
$[\Pk_1(K)]^2\cap\ker\divK$, and satisfies
\begin{equation}
 \|\bv-\mathcal R_K\bv\|_{0,K}
 \lesssim h_K|\bv|_{1,K}
 \qquad\forall\bv\in\mathcal Z(K).
 \label{eq:companion-local-estimates}
\end{equation}
\end{lemma}

\begin{proof}
Linearity and preservation of $[\Pk_1(K)]^2\cap\ker\divK$ follow directly from
\eqref{eq:poisson-companion}.
Let $I_K\bv$ be the $\mathrm{BDM}_2$ interpolant on the
fan.  The $H^1(K)$ regularity makes its moments well defined, and
uniform shape regularity gives
\[
 \|\bv-I_K\bv\|_{0,K}
 \lesssim h_K|\bv|_{1,K}.
\]
Since $I_K\bv\in\boldsymbol W_K(\bv)$, one has
$I_K\bv-\mathcal R_K\bv\in\boldsymbol W_K(\boldsymbol0)$.  Testing
the optimality condition of \eqref{eq:poisson-companion} with
$I_K\bv-\mathcal R_K\bv$, splitting the difference, and using
Poincar\'e's inequality and the interpolation estimate give
\begin{align*}
 \|I_K\bv-\mathcal R_K\bv\|_{0,K}^2
 &=\bigl(I_K\bv-\bv,I_K\bv-\mathcal R_K\bv\bigr)_K+\bigl(\bv-\Pi_1^{\nabla,K}\bv,
              I_K\bv-\mathcal R_K\bv\bigr)_K\\*
 &\lesssim h_K|\bv|_{1,K}\|I_K\bv-\mathcal R_K\bv\|_{0,K}.
\end{align*}
The triangle inequality proves \eqref{eq:companion-local-estimates}.
\end{proof}

\begin{lemma}[Exact pressure robustness]
\label{prop:closed-space-companion}
For every $\boldsymbol z_h\in\mathcal Z_h$,
$\mathcal R_h^{\rm c}\boldsymbol z_h\in\bm H(\divG;\Gamma_h)$ and
$\divG\mathcal R_h^{\rm c}\boldsymbol z_h=0$.  Moreover,
\begin{equation}
 \ell_h^Z(\bfv_h+\gradh\vartheta;\boldsymbol z_h)
 =\ell_h^Z(\bfv_h;\boldsymbol z_h)\qquad \forall \vartheta\in H^1(\Gamma_h).
 \label{eq:closed-space-exact-pr}
\end{equation}
\end{lemma}

\begin{proof}
For every $e\subset\partial K$, the normal-trace constraint
in \eqref{eq:poisson-companion} gives
\[
\begin{aligned}
 (\mathcal R_h^{\rm c}\boldsymbol z_h)\cdot\bn_K
 =(\mathcal R_K(\boldsymbol z_h|_K))\cdot\bn_K
 =\boldsymbol z_h|_K\cdot\bn_K.
\end{aligned}
\]
The last equality follows from
$\mathcal R_K(\boldsymbol z_h|_K)\in\boldsymbol W_K(\boldsymbol z_h|_K)$.
Thus $\mathcal R_h^{\rm c}\boldsymbol z_h$ inherits the normal-flux
conformity of $\boldsymbol z_h$ across the surface edges.
Membership in $\bm H(\operatorname{div};K)$ implies continuity of the normal
trace of $\mathcal R_h^{\rm c}\boldsymbol z_h$ across the interior fan edges.
The reconstructed field therefore belongs to $\bm H(\divG;\Gamma_h)$.

The local divergence constraint and normal-trace conformity give
$\divG\mathcal R_h^{\rm c}\boldsymbol z_h=0$.
Integration by parts on the closed surface proves
\eqref{eq:closed-space-exact-pr}.
\end{proof}

The reconstruction $\mathcal R_h^{\rm c}$ above supplies the load in the
virtual formulation \eqref{eq:topology-completed-stokes-method}.
Section~\ref{sec:explicit-representative} constructs a different,
piecewise-polynomial representative $E_h^\Sigma$ and gives an alternative
pressure-robust load.

\subsection{Stability}
\label{subsec:stability}

In this subsection, we prove that the discrete form is coercive on
$\mathcal Z_h$ with respect to the broken $H^1$ norm.  We first establish the
discrete energy bounds and then combine them with a Korn inequality.

\begin{lemma}[Bounds for the discrete energy]
\label{prop:uniform-local-stability}
Every $\bv_h\in\mathcal Z_h$ satisfies
\begin{equation}
 \|\eps_h\bv_h\|_{0,\Gamma_h}^2+\|\bv_h\|_{0,\Gamma_h}^2
 \lesssim a_h^Z(\bv_h,\bv_h)
 \lesssim \|\bv_h\|_{1,h}^2.
 \label{eq:uniform-local-form-bounds}
\end{equation}
Here $a_h^Z$ is defined by
\eqref{eq:divergence-free-bilinear-form}.  The constants depend only on the
constants in Assumption~\ref{ass:mesh}.
\end{lemma}

\begin{proof}
We prove the bounds elementwise, starting with boundedness.
Fix $\bv_h\in\mathcal Z_h$ and $K\in\Th$.  Set
$\bv:=\bv_h|_K$, $\boldsymbol p:=\Pi_1^{\nabla,K}\bv$, and
$\boldsymbol w:=(I-\Pi_1^{\nabla,K})\bv\in\mathcal Z(K)$.
By \eqref{eq:discrete-local-stokes-form},
\[
 a_{h,K}^V(\bv,\bv)
 =\|\eps_K\boldsymbol p\|_{0,K}^2+\|\boldsymbol p\|_{0,K}^2
   +S_K^\Sigma(\boldsymbol w,\boldsymbol w).
\]

The zero-sum normalization of the cumulative flux coordinates and polynomial
inverse estimates on the edges give
$S_K^\Sigma(\boldsymbol w,\boldsymbol w)\lesssim h_K^{-1}\|\boldsymbol w\|_{0,\partial K}^2$.
Since $\Pi_1^{\nabla,K}\boldsymbol w=0$, the mean normalization gives
$\|\boldsymbol w\|_{0,K}\lesssim h_K|\boldsymbol w|_{1,K}$, so
the scaled trace inequality yields
$S_K^\Sigma(\boldsymbol w,\boldsymbol w)\lesssim |\boldsymbol w|_{1,K}^2$.
Gradient orthogonality, the mean normalization, and $h_K\le1$ then give
\[
 a_{h,K}^V(\bv,\bv)
 \lesssim |\boldsymbol p|_{1,K}^2+\|\boldsymbol p\|_{0,K}^2
       +|\boldsymbol w|_{1,K}^2
 \lesssim |\bv|_{1,K}^2+\|\bv\|_{0,K}^2.
\]

For the lower bound, the decomposition $\bv=\boldsymbol p+\boldsymbol w$
and the mean normalization give
\[
 \|\eps_K\bv\|_{0,K}^2+\|\bv\|_{0,K}^2
 \lesssim\|\eps_K\boldsymbol p\|_{0,K}^2
       +\|\boldsymbol p\|_{0,K}^2+|\boldsymbol w|_{1,K}^2.
\]
Given $h_K\le1$, it suffices to prove
\[
 |\boldsymbol w|_{1,K}^2
 \lesssim S_K^\Sigma(\boldsymbol w,\boldsymbol w)
       +h_K^2\|\boldsymbol p\|_{0,K}^2.
\]
To this end, we control the tangential edge moments through the vertex
data.  Let $e\subset\partial K$ and
$a\in\partial e$.  
Since $\boldsymbol p$ is affine,
\eqref{eq:global-velocity-vertex-assembly} and
\eqref{eq:global-velocity-tangent-assembly} imply
\[
\begin{aligned}
 \frac1{|e|}\int_e\boldsymbol w\cdot\bt\,\mathrm ds
 =\frac12\sum_{a\in\partial e}
       (M_{a,K}^{-1}\boldsymbol w(a))\cdot\bt
 +\frac12\sum_{a\in\partial e}
       (M_{a,K}^{-1}\boldsymbol p(a)-\boldsymbol p(a))\cdot\bt.
\end{aligned}
\]
Formula~\eqref{eq:vertex-transform}, the geometry bounds,
and $\bnu_K\cdot\bt=0$ give
$
 |(M_{a,K}^{-1}\bx-\bx)\cdot\bt|
 \lesssim h_K^2|\bx|$, for all $\bx\in T_aK$.
The inverses $M_{a,K}^{-1}$ are uniformly bounded.  Using the edge
differences of $\xi_{K,i}$ to control the normal moments, polynomial norm
equivalence on each edge and an inverse estimate for $\boldsymbol p$ yield
\[
\begin{aligned}
 h_K^{-1}\|\boldsymbol w\|_{0,\partial K}^2
 &\lesssim h_K^{-2}\sum_i|\xi_{K,i}(\boldsymbol w)|^2
       +\sum_{a\in\mathcal V(K)}|\boldsymbol w(a)|^2
       +h_K^4\sum_{a\in\mathcal V(K)}|\boldsymbol p(a)|^2\\
 &\lesssim S_K^\Sigma(\boldsymbol w,\boldsymbol w)
       +h_K^2\|\boldsymbol p\|_{0,K}^2.
\end{aligned}
\]
Lemma~\ref{lem:local-velocity-estimates} now proves the required estimate
and hence the local lower bound.
Summing the upper and lower bounds over $K\in\Th$ proves
\eqref{eq:uniform-local-form-bounds}.
\end{proof}

\begin{theorem}[Korn estimate]
\label{thm:global-strain-control}
For sufficiently small $h$, the following estimate holds:
\begin{align}
 |\boldsymbol v_h|_{1,h}^2
 &\lesssim \|\eps_h\boldsymbol v_h\|_{0,\Gamma_h}^2
 +\|\boldsymbol v_h\|_{0,\Gamma_h}^2
 \qquad\forall \boldsymbol v_h\in\Sigmah.
 \label{eq:discrete-velocity-korn}
\end{align}
\end{theorem}

\begin{proof}
Fix $\boldsymbol v_h\in\Sigmah$.  Integration by parts, the edge
orientation convention, and $\jump{\boldsymbol v_h\cdot\bn}=0$ give the identity
\begin{equation*}
\begin{aligned}
 &2\|\eps_h\boldsymbol v_h\|_{0,\Gamma_h}^2
  -|\boldsymbol v_h|_{1,h}^2
  -\|\divh\boldsymbol v_h\|_{0,\Gamma_h}^2\\
 &\quad={}
 \sum_{e\in\Eh}\int_e
  \jump{\boldsymbol v_h\cdot\bt}
   \partial_{\bt}\avg{\boldsymbol v_h\cdot\bn}-\avg{\boldsymbol v_h\cdot\bn}
   \partial_{\bt}\jump{\boldsymbol v_h\cdot\bt} \,\mathrm{d}s.
\end{aligned}
\end{equation*}
The conormal traces and tangential jumps belong to $\Pk_2(e)$.
The polynomial inverse and trace inequalities, together with
\eqref{eq:velocity-tangential-jump-lifting}, give
\begin{equation*}
\begin{aligned}
 \|\avg{\boldsymbol v_h\cdot\bn}\|_{0,e}
 +h_e\|\partial_{\bt}
       \avg{\boldsymbol v_h\cdot\bn}\|_{0,e}
 &\lesssim h_e^{-1/2}
 \bigl(\|\boldsymbol v_h\|_{0,\omega_e}
       +h_e|\boldsymbol v_h|_{1,h,\omega_e}\bigr),\\
 \|\jump{\boldsymbol v_h\cdot\bt}\|_{0,e}
 +h_e\|\partial_{\bt}
       \jump{\boldsymbol v_h\cdot\bt}\|_{0,e}
 &\lesssim h_e^{3/2}
 \bigl(\|\boldsymbol v_h\|_{0,\omega_e}
       +h_e|\boldsymbol v_h|_{1,h,\omega_e}\bigr).
\end{aligned}
\end{equation*}
The Cauchy--Schwarz and bounded patch overlap now give
\[
\begin{aligned}  
 |\boldsymbol v_h|_{1,h}^2
 &\lesssim 2\|\eps_h\boldsymbol v_h\|_{0,\Gamma_h}^2 -\|\divh\boldsymbol v_h\|_{0,\Gamma_h}^2
 +\|\boldsymbol v_h\|_{0,\Gamma_h}^2
 +h^2|\boldsymbol v_h|_{1,h}^2\\
 &\lesssim \|\eps_h\boldsymbol v_h\|_{0,\Gamma_h}^2
  +\|\boldsymbol v_h\|_{0,\Gamma_h}^2
  +h^2|\boldsymbol v_h|_{1,h}^2.
\end{aligned} 
\]
For sufficiently small $h$, absorbing the last term proves
\eqref{eq:discrete-velocity-korn}.
\end{proof}

Combining Lemma~\ref{prop:uniform-local-stability} with the Korn estimate
\eqref{eq:discrete-velocity-korn} gives the following result.

\begin{theorem}[Coercivity]
\label{thm:kernel-H1-stability}
For the form $a_h^Z$ defined in
\eqref{eq:divergence-free-bilinear-form}, it holds that
\begin{equation}
 a_h^Z(\boldsymbol z_h,\boldsymbol z_h)
 \simeq \|\boldsymbol z_h\|_{1,h}^2
 \qquad\forall\boldsymbol z_h\in\mathcal Z_h.
 \label{eq:kernel-H1-stability}
\end{equation}
\end{theorem}

Finally, \eqref{eq:companion-local-estimates} gives
$\|\mathcal R_h^{\rm c}\boldsymbol z_h\|_{0,\Gamma_h}
\lesssim\|\boldsymbol z_h\|_{1,h}$, so the load is bounded in the broken
$H^1$ norm.  By \eqref{eq:kernel-H1-stability}, for sufficiently small $h$
problem~\eqref{eq:topology-completed-stokes-method} has a unique solution
$\bu_h\in\mathcal Z_h$ satisfying
$\|\bu_h\|_{1,h}\lesssim\|\bfv_h\|_{0,\Gamma_h}$.

\section{An explicit computational representative}
\label{sec:explicit-representative}
The virtual spaces of Sections~\ref{sec:geometry}--\ref{sec:method} are
defined implicitly.  Here we construct an explicit piecewise-polynomial
representative on a macroelement whose local functions share the same
boundary data as the virtual functions.  Evaluating both the energy and the
load through this representative gives a fully computable formulation on the
original virtual degrees of freedom that is exactly equivalent to a direct
macroelement Galerkin method. This is an alternative to the virtual
formulation \eqref{eq:topology-completed-stokes-method}: it uses a different
energy and load. This section establishes the well-posedness and exact
macroelement correspondence of the alternative, and
Section~\ref{sec:error-estimates} gives error estimates for both formulations.

\subsection{A macroelement representative of the complex}
\label{subsec:hct-complex}
Join $\bx_K^c$ to the vertices of $K$ to form the fan
$\widehat{\mathcal T}_K$ used in Subsection~\ref{subsec:local-companion}.
Split each fan triangle at its barycenter, and call the resulting
triangulation $\mathcal T_K^{\rm HCT}$. Under Assumption~\ref{ass:mesh},
these subtriangles are uniformly shape regular. Let
\begin{equation}\label{eq:hct-representative-space}
 \mathcal S_K^{\rm HCT}:=
 \{s\in C^1(\overline K):s|_T\in\Pk_3(T)
       \quad\forall T\in\mathcal T_K^{\rm HCT}\}\subset H^2(K)
\end{equation}
be the assembled Hsieh--Clough--Tocher space~\cite{percell1976clough}.
For $\phi\in\Phi(K)$, define the scalar representative by
\begin{equation}\label{eq:hct-scalar-representative}
 E_K^\Phi\phi:=
 \underset{\substack{s\in\mathcal S_K^{\rm HCT}\\
              s=\phi,\ \nabla_Ks=\nabla_K\phi
              \text{ on }\partial K}}
             {\operatorname{argmin}}\ |s|_{2,K}^2,
 \qquad \Phi^M(K):=E_K^\Phi\Phi(K).
\end{equation}
The cubic value trace and quadratic gradient trace of $\Phi(K)$ make the
constraints feasible by HCT unisolvence. The Hessian seminorm is a norm on
$\mathcal S_K^{\rm HCT}\cap H_0^2(K)$, so the interior extension is unique
and linear in the boundary data. Figure~\ref{fig:paired-local-dofs} shows
the boundary data and the locally eliminated HCT data.

For $\bv\in\Sigma(K)$, set $q=\divK\bv\in\Pk_0(K)$. Local exactness in
Lemma~\ref{lem:local-stokes-complex} gives a $\phi\in\Phi(K)$, unique up
to a constant, such that
\[
 \bv=\tfrac12(\bx-\bx_K^c)q+\curlK\phi.
\]
Define
\begin{equation}\label{eq:hct-velocity-representative}
 \begin{aligned}
 E_K^\Sigma\bv
  &:=\tfrac12(\bx-\bx_K^c)q+\curlK E_K^\Phi\phi,\\
 \Sigma^M(K)
  &:=\curlK\Phi^M(K)\oplus
       \operatorname{span}\{(\bx-\bx_K^c)/2\}.
 \end{aligned}
\end{equation}
This definition is independent of the constant in $\phi$.

\begin{lemma}[Local explicit representative]
\label{lem:hct-local-correspondence}
The maps $E_K^\Phi:\Phi(K)\to\Phi^M(K)$ and
$E_K^\Sigma:\Sigma(K)\to\Sigma^M(K)$ are isomorphisms preserving,
respectively, the full Cauchy and velocity boundary traces. They give a
commuting diagram of exact complexes:
\[
\begin{array}{cccccccccc}
0\longrightarrow&\mathbb R&\xrightarrow{\,\subset\,}&
\Phi(K)&\xrightarrow{\curlK}&
\Sigma(K)&\xrightarrow{\divK}&
\Pk_0(K)&\longrightarrow&0\\
&\big\downarrow{\scriptstyle I}&&
\big\downarrow{\scriptstyle E_K^\Phi}&&
\big\downarrow{\scriptstyle E_K^\Sigma}&&
\big\downarrow{\scriptstyle I}&&\\
0\longrightarrow&\mathbb R&\xrightarrow{\,\subset\,}&
\Phi^M(K)&\xrightarrow{\curlK}&
\Sigma^M(K)&\xrightarrow{\divK}&
\Pk_0(K)&\longrightarrow&0 .
\end{array}
\]
Moreover, $E_K^\Phi p=p$ for $p\in\Pk_3(K)$ and
$E_K^\Sigma\boldsymbol p=\boldsymbol p$ for
$\boldsymbol p\in[\Pk_1(K)]^2$.
\end{lemma}
\begin{proof}
The two scalar spaces are uniquely parameterized by the same Cauchy data.
Together with the decomposition preceding
\eqref{eq:hct-velocity-representative}, this shows that both maps are
trace-preserving isomorphisms.
The definitions give
$E_K^\Sigma\curlK=\curlK E_K^\Phi$ and
$\divK E_K^\Sigma=\divK$. Hence the diagram commutes and the lower
row is exact.

If $p\in\Pk_3(K)$ and $b\in\mathcal S_K^{\rm HCT}\cap H_0^2(K)$, two
integrations by parts give
$(D_K^2p,D_K^2b)_K=(\Delta_K^2p,b)_K=0$.
Since $p$ is admissible in \eqref{eq:hct-scalar-representative}, it is the
unique minimizer.  Thus $E_K^\Phi$ reproduces $\Pk_3(K)$.  Every affine
velocity is a constant-divergence lifting plus the curl of a quadratic
polynomial, so $E_K^\Sigma$ reproduces $[\Pk_1(K)]^2$.
In particular, $E_K^\Phi c=c$ for $c\in\mathbb R$.
\end{proof}

Assemble $\Phi_h^M$ and $\Sigma_h^M$ from $\Phi^M(K)$ and $\Sigma^M(K)$
using exactly the conditions \eqref{eq:global-stream-assembly} and
\eqref{eq:global-velocity-assembly}. Let $E_h^\Phi$ and $E_h^\Sigma$
act facewise. Since the local maps preserve the complete boundary traces,
they respect the assembly rules, while their facewise inverses are recovered
from the same boundary data. Hence the macro spaces have the same global
degrees of freedom as the virtual spaces. It then follows directly from the
definitions that:

\begin{lemma}[Global explicit representative]
\label{lem:hct-global-correspondence}
The global explicit spaces satisfy
$\Phi_h^M\subset H^1(\Gamma_h)$ and
$\Sigma_h^M\subset \bm H(\divG;\Gamma_h)$.
Moreover, the facewise maps define isomorphisms
$E_h^\Phi:\Phih\to\Phi_h^M$ and
$E_h^\Sigma:\Sigmah\to\Sigma_h^M$, and
\begin{equation}\label{eq:hct-global-commuting}
E_h^\Sigma\curlh=\curlh E_h^\Phi,
\qquad
\divh E_h^\Sigma=\divh.
\end{equation}
In particular, with
$\mathcal Z_h^M:=\{\bv_h^M\in\Sigma_h^M:\divh\bv_h^M=0\}$,
one has $E_h^\Sigma\mathcal Z_h=\mathcal Z_h^M$.
\end{lemma}

The representative also preserves the stability and approximation properties
needed for the error analysis. The proof of the following lemma is given
in Appendix~\ref{app:hct-error-estimates}.

\begin{lemma}[Stability and approximation of the representative]
\label{lem:hct-stability-approximation}
Under Assumption~\ref{ass:mesh}, the map $E_K^\Sigma$ defined by
\eqref{eq:hct-velocity-representative} satisfies, for every $\bv\in\Sigma(K)$,
\begin{equation}\label{eq:hct-local-stability}
 \|E_K^\Sigma\bv-\bv\|_{0,K}\lesssim h_K|\bv|_{1,K},
 \qquad
 \|E_K^\Sigma\bv\|_{1,K}\simeq\|\bv\|_{1,K}.
\end{equation}
For $\bv\in\bm H_t^2(\gamma)$, set
$I_\Sigma^M\bv:=E_h^\Sigma I_\Sigma\bv$. Then
\begin{equation}\label{eq:hct-interpolation-estimate}
 \|\breve\bv-I_\Sigma^M\bv\|_{0,\Gamma_h}
 +h\|\breve\bv-I_\Sigma^M\bv\|_{1,h}
 \lesssim h^2\|\bv\|_{2,\gamma}.
\end{equation}
If $\bv\in\mathcal Z(\gamma)$, then $I_\Sigma^M\bv\in\mathcal Z_h^M$.
\end{lemma}

\subsection{A macro-induced variational realization}
\label{subsec:hct-variational}

We now define a variant of the virtual formulation in
Subsection~\ref{subsec:global-discrete-problem}.  The space $\mathcal Z_h$
and its degrees of freedom remain unchanged, while the local projector,
stabilization, and load are evaluated through the explicit representative.

For a local velocity $\bv\in\Sigma(K)$, define the macro-induced
Stokes-energy projector $\Pi_{1,K}^{E,a}\bv\in[\Pk_1(K)]^2$ by
\begin{equation}\label{eq:hct-energy-projector}
 a_K(\Pi_{1,K}^{E,a}\bv,\boldsymbol p)
 =a_K(E_K^\Sigma\bv,\boldsymbol p)
 \qquad\forall\boldsymbol p\in[\Pk_1(K)]^2.
\end{equation}
Here $a_K$ is the exact Stokes energy in
\eqref{eq:continuous-face-stokes-form}.
Define the macro-induced stabilization by
$
 S_K^E(\bv,\boldsymbol z)
 :=a_K(E_K^\Sigma\bv,E_K^\Sigma\boldsymbol z).
$
The corresponding local virtual Stokes form is
\begin{equation}\label{eq:hct-virtual-local-form}
 a_{h,K}^E(\bv,\boldsymbol z)
 :=a_K(\Pi_{1,K}^{E,a}\bv,
          \Pi_{1,K}^{E,a}\boldsymbol z)+S_K^E((I-\Pi_{1,K}^{E,a})\bv,
          (I-\Pi_{1,K}^{E,a})\boldsymbol z).
\end{equation}
The form $a_{h,K}^E$ has the same projection--stabilization structure as
$a_{h,K}^V$ in \eqref{eq:discrete-local-stokes-form}, and it is also computable.
Since $E_K^\Sigma\bv$ is uniquely determined by the boundary trace of $\bv$,
it is available from the original velocity degrees of freedom. Thus
$\Pi_{1,K}^{E,a}\bv$, and hence \eqref{eq:hct-virtual-local-form}, can be
computed by polynomial integration on $\mathcal T_K^{\rm HCT}$.
For $\bv_h,\boldsymbol z_h\in\mathcal Z_h$, define
\begin{equation}\label{eq:hct-pulled-back-forms}
 a_h^E(\bv_h,\boldsymbol z_h)
  :=\sum_{K\in\Th}
     a_{h,K}^E(\bv_h|_K,\boldsymbol z_h|_K),\qquad
 \ell_h^E(\bfv_h;\boldsymbol z_h)
  :=(\bfv_h,E_h^\Sigma\boldsymbol z_h)_{\Gamma_h}.
\end{equation}
The macro-induced virtual problem is to find
$\bu_h^E\in\mathcal Z_h$ such that
\begin{equation}\label{eq:hct-virtual-method}
 a_h^E(\bu_h^E,\boldsymbol z_h)
 =\ell_h^E(\bfv_h;\boldsymbol z_h)
 \qquad\forall\boldsymbol z_h\in\mathcal Z_h.
\end{equation}

The explicit space also admits a direct Galerkin formulation.  Since
$\mathcal Z_h^M$ consists of piecewise quadratic velocities on the HCT
subtriangles, no projection or stabilization is needed.  For
$\bfv_h\in\bm L_{t,h}^2(\Gamma_h)$, find
$\bu_h^M\in\mathcal Z_h^M$ such that
\begin{equation}\label{eq:hct-direct-method}
 \sum_{K\in\Th}a_K(\bu_h^M|_K,\boldsymbol z_h^M|_K)
  =(\bfv_h,\boldsymbol z_h^M)_{\Gamma_h}
  \qquad\forall\boldsymbol z_h^M\in\mathcal Z_h^M.
\end{equation}
The macro spaces have the same boundary traces and assembly conditions as the
virtual spaces.  Hence the Korn argument of
Theorem~\ref{thm:global-strain-control} also applies to $\Sigma_h^M$, and,
for sufficiently small $h$,
\begin{equation}\label{eq:hct-coercivity}
 \sum_{K\in\Th}a_K(\boldsymbol z_h^M|_K,\boldsymbol z_h^M|_K)
 \simeq \|\boldsymbol z_h^M\|_{1,h}^2
 \qquad\forall\boldsymbol z_h^M\in\mathcal Z_h^M.
\end{equation}
In particular, the direct macroelement problem is well posed.  The next
result identifies it with the macro-induced virtual problem.
\begin{theorem}[Exact macro--virtual correspondence]
\label{thm:hct-variational-correspondence}
Problem~\eqref{eq:hct-virtual-method} has a unique virtual solution
$\bu_h^E$, and problem~\eqref{eq:hct-direct-method} has a unique macroelement
 solution $\bu_h^M$. With bases paired by $E_h^\Sigma$, the two problems have the same stiffness
 matrix and load vector. Their solutions satisfy
 \begin{equation}\label{eq:hct-solution-correspondence}
  \bu_h^M=E_h^\Sigma\bu_h^E.
 \end{equation}
\end{theorem}
\begin{proof}
For $\bv\in\Sigma(K)$, affine reproduction gives
$E_K^\Sigma\Pi_{1,K}^{E,a}\bv=\Pi_{1,K}^{E,a}\bv$.
Thus \eqref{eq:hct-energy-projector} implies
\[
 a_K\bigl(E_K^\Sigma(I-\Pi_{1,K}^{E,a})\bv,
           \boldsymbol p\bigr)=0
 \qquad\forall\boldsymbol p\in[\Pk_1(K)]^2.
\]
Applying this orthogonality to $\bv,\boldsymbol z\in\Sigma(K)$ gives
\begin{equation}\label{eq:hct-energy-decomposition}
 a_K(E_K^\Sigma\bv,E_K^\Sigma\boldsymbol z)
 =a_{h,K}^E(\bv,\boldsymbol z).
\end{equation}
The identity \eqref{eq:hct-energy-decomposition} and the load definition in
\eqref{eq:hct-pulled-back-forms} show that the virtual forms are the pullbacks of those in
\eqref{eq:hct-direct-method}. The isomorphism in
Lemma~\ref{lem:hct-global-correspondence} therefore identifies the two
problems, including their matrices and load vectors in paired bases.
Well-posedness of the macro problem gives uniqueness on the virtual side and
the solution relation \eqref{eq:hct-solution-correspondence}.
\end{proof}

Since $E_h^\Sigma\mathcal Z_h=\mathcal Z_h^M\subset
\bm H(\divG;\Gamma_h)$ is divergence-free, integration by parts on the
closed surface gives the same gradient invariance as in
Lemma~\ref{prop:closed-space-companion}.

\section{Error estimates}
\label{sec:error-estimates}
We first analyze the virtual formulation
\eqref{eq:topology-completed-stokes-method} of Section~\ref{sec:method}.
Bounds for the consistency errors arising from the
discrete load, the VEM forms, and the surface approximation yield
the broken $H^1$ and $L^2$ velocity error estimates. Error estimates are also
established for the macroelement formulations of
Section~\ref{sec:explicit-representative}. Throughout this section,
$\Gamma_h$ is assumed to satisfy Assumption~\ref{ass:mesh}.

We impose the following approximation property on the discrete load.

\begin{assumption}[Load approximation]
\label{ass:physical-data-interface}
For every tangential force $\bfv\in\bm L_t^2(\gamma)$, the datum
$\bfv_h\in\bm L_{t,h}^2(\Gamma_h)$ satisfies
\[
 \left|(
 \bfv_h,\boldsymbol w_h)_{\Gamma_h}
 -(\bfv,\widearc{\boldsymbol w_h})_\gamma\right|
 \lesssim h^2\|\bfv\|_{0,\gamma}
                 \|\boldsymbol w_h\|_{0,\Gamma_h}
 \qquad\forall\boldsymbol w_h\in\bm L_{t,h}^2(\Gamma_h).
\]
\end{assumption}
The canonical choice $\bfv_h=\bP_h\bfv^e$ satisfies this assumption.

For these estimates, we introduce two comparison forms.  The functional
$\ell(\boldsymbol v;\cdot)$ is used to compare the continuous and discrete
loads, while the broken bilinear form $a_h$ provides an intermediate form that
separates the VEM and geometric consistency errors.
For $\boldsymbol v\in\bm H_t^2(\gamma)$ and
$\boldsymbol z\in\bm H(\divg;\gamma)$, set
\[
 \ell(\boldsymbol v;\boldsymbol z)
 :=(-\bP\divg\eps_\gamma(\boldsymbol v)+\boldsymbol v,
     \boldsymbol z)_\gamma.
\]
In particular, the velocity solution $\bu$ of
\eqref{eq:continuous-stokes} satisfies
$\ell(\bu;\boldsymbol z)=(\bfv,\boldsymbol z)_\gamma$ for every
$\boldsymbol z\in\mathcal Z(\gamma)$.
For $\boldsymbol v,\boldsymbol z\in\bm H_{t,h}^1(\Gamma_h)$, define
\[
 a_h(\boldsymbol v,\boldsymbol z)
 :=\sum_{K\in\Th}\Bigl(
 (\eps_K\boldsymbol v,\eps_K\boldsymbol z)_K
 +(\boldsymbol v,\boldsymbol z)_K\Bigr).
\]

For $\bfv\in\bm L_t^2(\gamma)$, standard regularity for the surface Stokes
problem \eqref{eq:continuous-stokes}
\cite[cf. Lemma~2.1]{olshanskii2021infsup} gives
$(\bu,p)\in\bm H_t^2(\gamma)\times
(H^1(\gamma)\cap\mathring L^2(\gamma))$ and
\begin{equation}
 \|\bu\|_{2,\gamma}+\|p\|_{1,\gamma}
 \lesssim\|\bfv\|_{0,\gamma}.
 \label{eq:surface-stokes-regularity}
\end{equation}

\subsection{Consistency errors}

This subsection estimates, in turn, the errors due to the discrete load, the
VEM forms, and the surface geometry.

\begin{lemma}[Load consistency]
\label{lem:kernel-load-consistency}
Let $\bu\in\bm H_t^2(\gamma)$ be the velocity solution of
\eqref{eq:continuous-stokes} corresponding to
$\bfv\in\bm L_t^2(\gamma)$.
For all $\boldsymbol z_h\in\mathcal Z_h$ and
$\boldsymbol z\in\mathcal Z(\gamma)\cap\bm H_t^2(\gamma)$, with
$\boldsymbol z_I:=I_\Sigma\boldsymbol z$, we have
\begin{subequations}
\begin{align}
 \left|
 \ell_h^Z(\bfv_h;\boldsymbol z_h)
 -\ell(\bu;\widearc{\boldsymbol z_h})
 \right|
 &\lesssim \left(h\|\bu\|_{2,\gamma}+h^2\|\bfv\|_{0,\gamma}\right)
           \|\boldsymbol z_h\|_{1,h},
 \label{eq:complete-kernel-load-consistency}\\
 \left|
 \ell_h^Z(\bfv_h;\boldsymbol z_I)
 -\ell(\bu;\widearc{\boldsymbol z_I})
 \right|
 &\lesssim h^2\|\bfv\|_{0,\gamma}
                  \|\boldsymbol z\|_{2,\gamma}.
 \label{eq:paired-kernel-load-consistency}
\end{align}
\end{subequations}
\end{lemma}

\begin{proof}
For $\boldsymbol z_h\in\mathcal Z_h$, it holds that $\widearc{\mathcal R_h^{\rm c}\boldsymbol z_h}$ belongs to $\bm H(\divg;\gamma)$ and is divergence-free.
Integration by parts therefore yields
\[
 \ell_h^Z(\bfv_h;\boldsymbol z_h)
 -\ell(\bu;\widearc{\boldsymbol z_h})
 =(\bfv_h,\mathcal R_h^{\rm c}\boldsymbol z_h)_{\Gamma_h}
      -(\bfv,\widearc{\mathcal R_h^{\rm c}\boldsymbol z_h})_\gamma
 +\ell\bigl(\bu;
       \widearc{\mathcal R_h^{\rm c}\boldsymbol z_h}
       -\widearc{\boldsymbol z_h}\bigr).
\]
Assumption~\ref{ass:physical-data-interface},
the definition of $\ell$, and the Piola norm equivalence give
\begin{equation}\label{eq:load-1}
\begin{aligned}
 \left|\ell_h^Z(\bfv_h;\boldsymbol z_h)
       -\ell(\bu;\widearc{\boldsymbol z_h})\right|
 &\lesssim h^2\|\bfv\|_{0,\gamma}
       \|\mathcal R_h^{\rm c}\boldsymbol z_h\|_{0,\Gamma_h}
       +\|\bu\|_{2,\gamma}
       \|\mathcal R_h^{\rm c}\boldsymbol z_h-\boldsymbol z_h\|_{0,\Gamma_h}\\
 &\lesssim h^2\|\bfv\|_{0,\gamma}\|\boldsymbol z_h\|_{1,h}
       +\|\bu\|_{2,\gamma}
       \|\mathcal R_h^{\rm c}\boldsymbol z_h-\boldsymbol z_h\|_{0,\Gamma_h}.
\end{aligned}
\end{equation}
The last step uses \eqref{eq:companion-local-estimates}.
The same local reconstruction estimate gives
\[
\begin{aligned}
 \|\mathcal R_h^{\rm c}\boldsymbol z_h-\boldsymbol z_h\|_{0,K}
 =\|(\mathcal R_K-I)\boldsymbol z_h\|_{0,K}
 \lesssim h_K|\boldsymbol z_h|_{1,K}.
\end{aligned}
\]
Squaring and summing over $K$.
Together with \eqref{eq:load-1}, this proves
\eqref{eq:complete-kernel-load-consistency}.

For $\boldsymbol z\in\mathcal Z(\gamma)\cap\bm H_t^2(\gamma)$,
Lemma~\ref{lem:commuting-interpolation} gives
$\boldsymbol z_I\in\mathcal Z_h$.  On each $K$, choose
$\boldsymbol p_K:=\Pi_1^{\nabla,K}\breve{\boldsymbol z}_K$.
The difference $\boldsymbol z_I|_K-\boldsymbol p_K$ belongs to
$\mathcal Z(K)$, since $\divK\boldsymbol p_K=0$.
Reproduction of divergence-free affine polynomials by $\mathcal R_K$,
\eqref{eq:companion-local-estimates}, and
\eqref{eq:general-velocity-interpolation-estimate} give
\[
\begin{aligned}
 \|\mathcal R_h^{\rm c}\boldsymbol z_I
       -\boldsymbol z_I\|_{0,K}
 &=\|(\mathcal R_K-I)
       (\boldsymbol z_I-\boldsymbol p_K)\|_{0,K}
 \lesssim h_K |\boldsymbol z_I-\boldsymbol p_K|_{1,K}\\
 &\le h_K \left( |\boldsymbol z_I-\breve{\boldsymbol z}|_{1,K}
      +|\breve{\boldsymbol z}-\boldsymbol p_K|_{1,K}\right) 
 \lesssim h_K^2\|\boldsymbol z\|_{2,\omega_K^\gamma}.
\end{aligned}
\]
Squaring and summing over $K$, with the bounded overlap of the patches,
gives
$\|\mathcal R_h^{\rm c}\boldsymbol z_I-\boldsymbol z_I\|_{0,\Gamma_h}
\lesssim h^2\|\boldsymbol z\|_{2,\gamma}$.
The interpolation estimate also gives
$\|\boldsymbol z_I\|_{1,h}\lesssim\|\boldsymbol z\|_{2,\gamma}$.
Applying \eqref{eq:load-1} with $\boldsymbol z_h=\boldsymbol z_I$ and using
$\|\bu\|_{2,\gamma}\lesssim\|\bfv\|_{0,\gamma}$ from
\eqref{eq:surface-stokes-regularity} proves
\eqref{eq:paired-kernel-load-consistency}.
\end{proof}

\begin{lemma}[VEM consistency]
\label{lem:complete-kernel-consistency}
Let $\boldsymbol v\in\mathcal Z(\gamma)\cap\bm H_t^2(\gamma)$.  Then
\begin{subequations}
\begin{align}
 \left|
 a_h^Z(\boldsymbol v_I,\boldsymbol z_h)
 -a_h(\breve{\boldsymbol v},\boldsymbol z_h)
 \right|
 &\lesssim h\|\boldsymbol v\|_{2,\gamma}
                 \|\boldsymbol z_h\|_{1,h}
 &&\forall\boldsymbol z_h\in\mathcal Z_h,
 \label{eq:complete-kernel-vem-consistency}\\
 \left|
 a_h^Z(\boldsymbol v_I,\boldsymbol z_I)
 -a_h(\boldsymbol v_I,\boldsymbol z_I)
 \right|
 &\lesssim h^2\|\boldsymbol v\|_{2,\gamma}
                  \|\boldsymbol z\|_{2,\gamma}
 &&\forall\boldsymbol z\in
       \mathcal Z(\gamma)\cap\bm H_t^2(\gamma).
 \label{eq:paired-kernel-vem-consistency}
\end{align}
\end{subequations}
where $\boldsymbol v_I:=I_\Sigma\boldsymbol v$ and
$\boldsymbol z_I:=I_\Sigma\boldsymbol z$.
\end{lemma}

\begin{proof}
For $\boldsymbol z_h\in\mathcal Z_h$,
\eqref{eq:discrete-local-stokes-form} and
\eqref{eq:projected-stream-velocity-means} give
$
 a_h^Z(\boldsymbol v_I,\boldsymbol z_h)
 -a_h(\boldsymbol v_I,\boldsymbol z_h)=I_1+I_2+I_3,
$
where
\[
\begin{aligned}
 I_1&:=\sum_{K\in\Th}S_K^\Sigma
       \bigl((I-\Pi_1^{\nabla,K})\boldsymbol v_I,
             (I-\Pi_1^{\nabla,K})\boldsymbol z_h\bigr),\\
 I_2&:=-\sum_{K\in\Th}
       \bigl((I-\Pi_0^K)\eps_K\boldsymbol v_I,
             (I-\Pi_0^K)\eps_K\boldsymbol z_h\bigr)_K,\\
 I_3&:=\sum_{K\in\Th}
       \left[\bigl(\Pi_1^{\nabla,K}\boldsymbol v_I,
                    \Pi_1^{\nabla,K}\boldsymbol z_h\bigr)_K
             -(\boldsymbol v_I,\boldsymbol z_h)_K\right].
\end{aligned}
\]

For $I_1$, gradient best approximation, the $H^1$-seminorm stability of $\Pi_1^K$,
polynomial approximation, and
\eqref{eq:general-velocity-interpolation-estimate} give
\[
\begin{aligned}
 |(I-\Pi_1^{\nabla,K})\boldsymbol v_I|_{1,K}
 \le |(I-\Pi_1^K)\boldsymbol v_I|_{1,K}
 \lesssim |\boldsymbol v_I-\breve{\boldsymbol v}|_{1,K}
              +h_K|\breve{\boldsymbol v}|_{2,K}
 \lesssim h_K\|\boldsymbol v\|_{2,\omega_K^\gamma}.
\end{aligned}
\]
The same estimate holds for $\boldsymbol z_I$.
For general $\boldsymbol z_h\in\mathcal Z_h$, gradient orthogonality gives
$|(I-\Pi_1^{\nabla,K})\boldsymbol z_h|_{1,K}\le|\boldsymbol z_h|_{1,K}$.
The stabilization bound in the proof of
Lemma~\ref{prop:uniform-local-stability}, Cauchy--Schwarz, and bounded patch
overlap yield
\[
\begin{alignedat}{2}
 |I_1|\lesssim h\|\boldsymbol v\|_{2,\gamma}|\boldsymbol z_h|_{1,h}
\quad\forall\boldsymbol z_h\in\mathcal Z_h,\qquad
 |I_1|\lesssim h^2\|\boldsymbol v\|_{2,\gamma}\|\boldsymbol z\|_{2,\gamma}
\quad\text{if }\boldsymbol z_h=\boldsymbol z_I.
\end{alignedat}
\]

For $I_2$, \eqref{eq:projected-stream-velocity-means},
Cauchy--Schwarz, and the same residual estimates give the needed estimates.

For $I_3$, the mean identity in
\eqref{eq:hessian-projector-kernel} allows constant means to be
subtracted in the velocity pairings.  Cauchy--Schwarz, Poincar\'e's inequality,
and gradient orthogonality then give
\begin{equation}\label{eq:local-mass-consistency}
\begin{aligned}
 &\bigl|(\Pi_1^{\nabla,K}\boldsymbol v_I,
          \Pi_1^{\nabla,K}\boldsymbol z_h)_K
         -(\boldsymbol v_I,\boldsymbol z_h)_K\bigr|\\
 &\quad\le \|(I-\Pi_1^{\nabla,K})\boldsymbol v_I\|_{0,K}
       \|\Pi_1^{\nabla,K}\boldsymbol z_h
          -\Pi_0^K\boldsymbol z_h\|_{0,K}\\
 &\qquad+\|\boldsymbol v_I-\Pi_0^K\boldsymbol v_I\|_{0,K}
       \|(I-\Pi_1^{\nabla,K})\boldsymbol z_h\|_{0,K}\\
 &\quad\lesssim h_K\bigl(
       \|(I-\Pi_1^{\nabla,K})\boldsymbol v_I\|_{0,K}
       |\boldsymbol z_h|_{1,K}+|\boldsymbol v_I|_{1,K}
       \|(I-\Pi_1^{\nabla,K})\boldsymbol z_h\|_{0,K}\bigr)\\
 &\quad\lesssim h_K^2|\boldsymbol v_I|_{1,K}|\boldsymbol z_h|_{1,K}.
\end{aligned}
\end{equation}
Since $|\boldsymbol v_I|_{1,K}
\lesssim\|\boldsymbol v\|_{2,\omega_K^\gamma}$ by
\eqref{eq:general-velocity-interpolation-estimate}, summation gives
$
 |I_3|\lesssim h^2\|\boldsymbol v\|_{2,\gamma}|\boldsymbol z_h|_{1,h}.
$
For $\boldsymbol z_h=\boldsymbol z_I$, interpolation stability bounds
the right-hand side by
$Ch^2\|\boldsymbol v\|_{2,\gamma}\|\boldsymbol z\|_{2,\gamma}$.

Finally, \eqref{eq:general-velocity-interpolation-estimate} gives
\[
 |a_h(\boldsymbol v_I-\breve{\boldsymbol v},\boldsymbol z_h)|
 \lesssim h\|\boldsymbol v\|_{2,\gamma}\|\boldsymbol z_h\|_{1,h}.
\]
Adding this interpolation term to $I_1+I_2+I_3$ proves
\eqref{eq:complete-kernel-vem-consistency}.
The bounds for $\boldsymbol z_h=\boldsymbol z_I$ prove
\eqref{eq:paired-kernel-vem-consistency} directly.
\end{proof}

We estimate the geometric consistency errors using a Piola geometric
estimate and a weak interpolation estimate.  For broken arguments,
$a_\gamma$ is understood facewise.  The proof of the following lemma is
given in Appendix~\ref{app:piola-geometric-estimate}.

\begin{lemma}[Piola geometric estimates]
\label{lem:piola-geometric-comparison}
For every $K\in\Th$ and
$\boldsymbol v\in\bm H_t^1(K^\gamma)$, the Piola pullback satisfies
\begin{subequations}
\begin{align}
 \left|
 \eps_K\breve{\boldsymbol v}
 -(\eps_\gamma(\boldsymbol v))^e
 \right|
 &\lesssim h\left(
 |\nabla_K\breve{\boldsymbol v}|+|\breve{\boldsymbol v}|
 \right)
 &&\text{a.e. in }K,
 \label{eq:pointwise-piola-strain}\\
 \left\|
 \eps_K\breve{\boldsymbol v}
 -(\eps_\gamma(\boldsymbol v))^e
 \right\|_{0,K}
 &\lesssim h\|\boldsymbol v\|_{1,K^\gamma}.
 \label{eq:local-piola-strain}
\end{align}
For all $\boldsymbol w_h,\boldsymbol z_h\in\bm H_{t,h}^1(\Gamma_h)$,
\begin{align}
 \left|
 a_\gamma(\widearc{\boldsymbol w_h},\widearc{\boldsymbol z_h})
 -a_h(\boldsymbol w_h,\boldsymbol z_h)
 \right|
 &\lesssim h\|\boldsymbol w_h\|_{1,h}
                 \|\boldsymbol z_h\|_{1,h},
 \label{eq:piola-form-comparison}\\
 \left|
 (\widearc{\boldsymbol w_h},\widearc{\boldsymbol z_h})_\gamma
 -(\boldsymbol w_h,\boldsymbol z_h)_{\Gamma_h}
 \right|
 &\lesssim h^2\|\boldsymbol w_h\|_{0,\Gamma_h}
                  \|\boldsymbol z_h\|_{0,\Gamma_h}.
 \label{eq:piola-mass-comparison}
\end{align}
For $\boldsymbol v,\boldsymbol z\in
\mathcal Z(\gamma)\cap\bm H_t^2(\gamma)$, we also have
\begin{equation}
 \left|
 a_h(\breve{\boldsymbol v},\breve{\boldsymbol z})
 -a_\gamma(\boldsymbol v,\boldsymbol z)
 \right|
 \lesssim h^2\|\boldsymbol v\|_{2,\gamma}
                  \|\boldsymbol z\|_{2,\gamma}.
 \label{eq:paired-kernel-smooth-consistency}
\end{equation}
\end{subequations}
\end{lemma}

\begin{lemma}[Weak interpolation estimate]
\label{lem:weak-strain-interpolation}
Let $\bm A\in H^1(\gamma;\mathbb R^{3\times3})$ be symmetric and tangential.
For $\boldsymbol z\in\mathcal Z(\gamma)\cap\bm H_t^2(\gamma)$ and
$\boldsymbol z_I:=I_\Sigma\boldsymbol z$, we have
\begin{equation}
 \left|
 (\eps_h(\boldsymbol z_I-\breve{\boldsymbol z}),\bm A^e)_{\Gamma_h}
 \right|
 \lesssim h^2\|\bm A\|_{1,\gamma}\|\boldsymbol z\|_{2,\gamma}.
 \label{eq:weak-strain-interpolation}
\end{equation}
\end{lemma}

\begin{proof}
Set $\boldsymbol e_h:=\boldsymbol z_I-\breve{\boldsymbol z}$ and
$\bar{\bm A}_e:=|e|^{-1}\int_e\bm A^e\,\mathrm{d}s$.
For $e=\partial K\cap\partial L$, write
$\jump{\bn}:=\bn_K+\bn_L$ and
$\avg{\bn}:=\tfrac12(\bn_K-\bn_L)$.
Since $\bm A$ is symmetric, elementwise integration by parts and
decomposition into tangential and conormal edge components give
$(\eps_h\boldsymbol e_h,\bm A^e)_{\Gamma_h}=\sum_{j=1}^5 I_j$, where
\[
\begin{aligned}
 I_1&:=-\sum_{K\in\Th}
   (\boldsymbol e_h,\divK(\bP_K\bm A^e\bP_K))_K,\\
 I_2&:=\sum_{K\in\Th}\sum_{e\subset\partial K}
   \int_e(\boldsymbol e_h\cdot\bn_K)
   \bn_K^T(\bm A^e-\bar{\bm A}_e)\bn_K\,\mathrm{d}s,\\
 I_3&:=\sum_{e\in\Eh}\int_e
   \jump{\boldsymbol z_I\cdot\bt}\,
   \bt^T(\bm A^e-\bar{\bm A}_e)\avg{\bn}\,\mathrm{d}s,\\
 I_4&:=-\sum_{e\in\Eh}\int_e
   \jump{\breve{\boldsymbol z}\cdot\bt}\,
   \bt^T\bm A^e\avg{\bn}\,\mathrm{d}s,\\
 I_5&:=\sum_{e\in\Eh}\int_e
   \avg{\boldsymbol e_h\cdot\bt}\,
   \bt^T\bm A^e\jump{\bn}\,\mathrm{d}s.
\end{aligned}
\]
In $I_2$ and $I_3$, we have subtracted the constant tensor
$\bar{\bm A}_e$.  This is justified by the zero edge means of the
conormal error in $I_2$ and the tangential jump of $\boldsymbol z_I$ in
$I_3$, which follow from \eqref{eq:interpolation-edge-flux-data} and
\eqref{eq:global-velocity-tangent-assembly}, respectively.
Estimates \eqref{eq:general-velocity-interpolation-estimate} and
\eqref{eq:edge-piola-tangential-comparison}, the scaled trace and
Poincar\'e inequalities, and bounded patch overlap give
\[
 |I_1|+|I_2|+|I_3|
 \lesssim h^2\|\bm A\|_{1,\gamma}\|\boldsymbol z\|_{2,\gamma}.
\]
For $I_5$, the common edge tangent and \eqref{eq:geometry-estimates} give
$|\bP\jump{\bn}|\lesssim h_e^2$.
Tangentiality of $\bm A$, the interpolation estimate, and the scaled
trace inequality therefore yield
$|I_5|\lesssim h^2\|\bm A\|_{1,\gamma}\|\boldsymbol z\|_{2,\gamma}$.

The term $I_4$ is independent of the discrete space.
Lemma~\ref{lem:geometric-edge-error} in
Appendix~\ref{app:geometric-edge-error} gives
$|I_4|\lesssim h^2\|\bm A\|_{1,\gamma}\|\boldsymbol z\|_{2,\gamma}$.
Combining these bounds proves \eqref{eq:weak-strain-interpolation}.
\end{proof}

\begin{remark}[Role of the edge correction]
\label{rem:uncorrected-surface-consistency}
For velocities $\widetilde{\boldsymbol z}_h$ in the uncorrected space
$\widetilde\Sigma_h$ defined in
\eqref{eq:uncorrected-global-velocity-space}, the tangential edge traces
are linear and the tangential jump need not have zero mean.
In the preceding proof, the zero-mean property is used in $I_3$ to
subtract $\bar{\bm A}_e$.  For $\widetilde{\boldsymbol z}_h$, this step
leaves the additional contribution
\[
 \sum_{e\in\Eh}\int_e
 \jump{\widetilde{\boldsymbol z}_h\cdot\bt}\,
 \bt^T\bar{\bm A}_e\avg{\bn}\,\mathrm{d}s.
\]
This term admits only an $O(h)$
bound.  Thus the second-order weak estimate
\eqref{eq:weak-strain-interpolation}, needed for the $L^2$ error analysis,
is not guaranteed for these uncorrected velocities.
The shared edge constraint \eqref{eq:global-velocity-tangent-assembly}
eliminates this additional contribution to $I_3$.
\end{remark}

The preceding estimates yield the geometric consistency bounds.

\begin{lemma}[Geometric consistency]
\label{lem:surface-consistency}
Let $\boldsymbol v\in\mathcal Z(\gamma)\cap\bm H_t^2(\gamma)$.  Then
\begin{subequations}
 \begin{align}
 \left|
 a_h(\breve{\boldsymbol v},\boldsymbol z_h)
 -\ell(\boldsymbol v;\widearc{\boldsymbol z_h})
 \right|
 &\lesssim h\|\boldsymbol v\|_{2,\gamma}
                 \|\boldsymbol z_h\|_{1,h}
 &&\forall\boldsymbol z_h\in\Sigmah,
 \label{eq:complete-kernel-surface-consistency}\\
 \left|
 a_h(\breve{\boldsymbol v},\boldsymbol z_I)
 -\ell(\boldsymbol v;\widearc{\boldsymbol z_I})
 \right|
 &\lesssim h^2\|\boldsymbol v\|_{2,\gamma}
                  \|\boldsymbol z\|_{2,\gamma}
 &&\forall\boldsymbol z\in
       \mathcal Z(\gamma)\cap\bm H_t^2(\gamma),
 \label{eq:paired-kernel-strong-consistency}
\end{align}
\end{subequations}
where $\boldsymbol z_I:=I_\Sigma\boldsymbol z$.
\end{lemma}

\begin{proof}
Set $\bm A=\eps_\gamma(\boldsymbol v)$.  For
$\boldsymbol z_h\in\Sigmah$, facewise integration by parts,
\eqref{eq:global-velocity-assembly}, the geometry estimates, and the scaled
trace inequality give
\[
 \left|
 (\eps_h\boldsymbol z_h,\bm A^e)_{\Gamma_h}
 +(\boldsymbol z_h,\divh\bm A^e)_{\Gamma_h}
 \right|
 \lesssim h\|\bm A\|_{1,\gamma}\|\boldsymbol z_h\|_{1,h}.
\]
Moreover, \eqref{eq:surface-piola} and \eqref{eq:geometry-estimates} imply
\[
 \left|
 (\boldsymbol z_h,\divh\bm A^e)_{\Gamma_h}
 -(\widearc{\boldsymbol z_h},\divg\bm A)_\gamma
 \right|
 \lesssim h\|\bm A\|_{1,\gamma}\|\boldsymbol z_h\|_{0,\Gamma_h},
\]
where tensor divergence is taken rowwise.  These estimates,
\eqref{eq:local-piola-strain}, and
\eqref{eq:piola-mass-comparison} prove
\eqref{eq:complete-kernel-surface-consistency}.

For \eqref{eq:paired-kernel-strong-consistency}, insert
$\breve{\boldsymbol z}$ and use
$\ell(\boldsymbol v;\boldsymbol z)=a_\gamma(\boldsymbol v,\boldsymbol z)$:
\[
\begin{aligned}
 a_h(\breve{\boldsymbol v},\boldsymbol z_I)
 -\ell(\boldsymbol v;\widearc{\boldsymbol z_I})
 =\big(a_h(\breve{\boldsymbol v},\breve{\boldsymbol z})
       -a_\gamma(\boldsymbol v,\boldsymbol z)\big)
+\big( a_h(\breve{\boldsymbol v},
                    \boldsymbol z_I-\breve{\boldsymbol z})
       -\ell(\boldsymbol v;\widearc{\boldsymbol z_I}-\boldsymbol z)\big).
\end{aligned}
\]
The first difference is bounded by
\eqref{eq:paired-kernel-smooth-consistency}.  For the remaining strain term,
write
\[
\begin{aligned}
 (\eps_h\breve{\boldsymbol v},
     \eps_h(\boldsymbol z_I-\breve{\boldsymbol z}))_{\Gamma_h}
=(\eps_h\breve{\boldsymbol v}-\bm A^e,
          \eps_h(\boldsymbol z_I-\breve{\boldsymbol z}))_{\Gamma_h}
       +(\bm A^e,\eps_h(\boldsymbol z_I-\breve{\boldsymbol z}))_{\Gamma_h}.
\end{aligned}
\]
The first pairing is bounded by \eqref{eq:local-piola-strain} and
\eqref{eq:general-velocity-interpolation-estimate}, and the second by
\eqref{eq:weak-strain-interpolation}.  Both are bounded by
$h^2\|\boldsymbol v\|_{2,\gamma}\|\boldsymbol z\|_{2,\gamma}$.
The mass and load terms satisfy the same bound by the $L^2$ interpolation
estimate and \eqref{eq:surface-piola-norm-equivalence}.  This proves
\eqref{eq:paired-kernel-strong-consistency}.
\end{proof}

\subsection{Broken \texorpdfstring{$H^1$}{H1} and \texorpdfstring{$L^2$}{L2} error estimates}

\begin{theorem}[Broken $H^1$ error]
\label{thm:arbitrary-topology-H1-error}
Let $\bu\in\bm H_t^2(\gamma)$ and $\bu_h\in\mathcal Z_h$ solve
\eqref{eq:continuous-stokes} and \eqref{eq:topology-completed-stokes-method},
respectively.  Suppose that
Assumptions~\ref{ass:mesh} and \ref{ass:physical-data-interface} hold.  Then
\begin{equation}
 \|\breve\bu-\bu_h\|_{1,h}
 \lesssim h\|\bu\|_{2,\gamma}+h^2\|\bfv\|_{0,\gamma}.
 \label{eq:arbitrary-topology-H1-error}
\end{equation}
In particular, if $\bfv=\gradg p$ for $p\in H^1(\gamma)$, then
$\bu=\boldsymbol0$ and
$\|\bu_h\|_{1,h}\lesssim h^2\|\gradg p\|_{0,\gamma}$.
\end{theorem}

\begin{proof}
Lemma~\ref{lem:commuting-interpolation} gives
$I_\Sigma\bu\in\mathcal Z_h$.  Strang's lemma and coercivity
\eqref{eq:kernel-H1-stability} give
\[
 \|\breve\bu-\bu_h\|_{1,h}
 \lesssim \|\breve\bu-I_\Sigma\bu\|_{1,h}
 +\sup_{0\ne\boldsymbol z_h\in\mathcal Z_h}
 \frac{|a_h^Z(I_\Sigma\bu,\boldsymbol z_h)
          -\ell_h^Z(\bfv_h;\boldsymbol z_h)|}
      {\|\boldsymbol z_h\|_{1,h}}.
\]
The residual decomposes as
\[
 \begin{aligned}
a_h^Z(I_\Sigma\bu,\boldsymbol z_h)
  -\ell_h^Z(\bfv_h;\boldsymbol z_h)
 &=
 \bigl(a_h^Z(I_\Sigma\bu,\boldsymbol z_h)
       -a_h(\breve\bu,\boldsymbol z_h)\bigr)
 +
 \bigl(a_h(\breve\bu,\boldsymbol z_h)
       -\ell(\bu;\widearc{\boldsymbol z_h})\bigr)\\
 &\qquad+
 \bigl(\ell(\bu;\widearc{\boldsymbol z_h})
       -\ell_h^Z(\bfv_h;\boldsymbol z_h)\bigr).
 \end{aligned}
\]
Equations~\eqref{eq:complete-kernel-vem-consistency},
\eqref{eq:complete-kernel-surface-consistency}, and
\eqref{eq:complete-kernel-load-consistency} control the three terms,
respectively.  Together with
\eqref{eq:general-velocity-interpolation-estimate}, these estimates prove
\eqref{eq:arbitrary-topology-H1-error}.
For a gradient load, the continuous problem has the velocity solution
$\bu=\boldsymbol0$, which gives the stated particular case.
\end{proof}

\begin{theorem}[$L^2$ error]
\label{thm:arbitrary-topology-lower-order-error}
Let $\bu\in\bm H_t^2(\gamma)$ and $\bu_h\in\mathcal Z_h$ solve
\eqref{eq:continuous-stokes} and \eqref{eq:topology-completed-stokes-method},
respectively.  Suppose that
Assumptions~\ref{ass:mesh} and \ref{ass:physical-data-interface} hold.  Then
\begin{equation}
 \|\breve\bu-\bu_h\|_{0,\Gamma_h}
 \lesssim h^2\|\bfv\|_{0,\gamma}.
 \label{eq:arbitrary-topology-velocity-L2-error}
\end{equation}
\end{theorem}

\begin{proof}
Set $\boldsymbol e:=\bu-\widearc{\bu_h}$.  Equations
\eqref{eq:surface-piola-div}--\eqref{eq:surface-piola-normal-trace} give
$\boldsymbol e\in\bm H(\divg;\gamma)$,
$\breve{\boldsymbol e}=\breve\bu-\bu_h$, and
$\divg\boldsymbol e=0$.
Let
$(\boldsymbol z,p^\ast)\in\mathcal Z(\gamma)\times\mathring H^1(\gamma)$
satisfy
\[
 -\bP\divg\eps_\gamma(\boldsymbol z)+\boldsymbol z+\gradg p^\ast
 =\boldsymbol e.
\]
The regularity estimate \eqref{eq:surface-stokes-regularity} gives
$
 \|\boldsymbol z\|_{2,\gamma}
 \lesssim\|\boldsymbol e\|_{0,\gamma}.$

Since $\divg\boldsymbol e=0$, testing the equation with
$\boldsymbol e$ eliminates the pressure term.  With
$\boldsymbol z_I:=I_\Sigma\boldsymbol z\in\mathcal Z_h$, it holds that
\begin{align*}
 \|\boldsymbol e\|_{0,\gamma}^2 = \ell(\boldsymbol z;\boldsymbol e)
 ={}&\bigl(
 \ell(\boldsymbol z;\boldsymbol e)
 -a_h(\boldsymbol z_I,\breve{\boldsymbol e})
 \bigr)
 +\bigl(
 a_h(\breve\bu,\boldsymbol z_I)
 -\ell(\bu;\widearc{\boldsymbol z_I})
 \bigr)\\
 &+\bigl(
 \ell(\bu;\widearc{\boldsymbol z_I})
 -\ell_h^Z(\bfv_h;\boldsymbol z_I)
 \bigr)
 +\bigl(
 a_h^Z(\boldsymbol z_I,\bu_h)
 -a_h(\boldsymbol z_I,\bu_h)
 \bigr).
\end{align*}
Denote the four terms by $I_1,I_2,I_3,I_4$ in order.
Equations~\eqref{eq:paired-kernel-strong-consistency} and
\eqref{eq:paired-kernel-load-consistency} give
\[
 |I_2|+|I_3|
 \lesssim h^2\left(\|\bu\|_{2,\gamma}+\|\bfv\|_{0,\gamma}\right)
                  \|\boldsymbol z\|_{2,\gamma}.
\]
Let $\bu_I = I_\Sigma\bu$ and $\boldsymbol e_h:=\bu_I-\bu_h\in\mathcal Z_h$. By VEM consistency error estimates 
\eqref{eq:paired-kernel-vem-consistency}, \eqref{eq:complete-kernel-vem-consistency}, and the interpolation estimate,
it holds that
\[
\begin{aligned}
|I_4|
 &=\left|
 a_h^Z(\boldsymbol z_I,\bu_I)
 -a_h(\boldsymbol z_I,\bu_I)-
 \bigl(a_h^Z(\boldsymbol z_I,\boldsymbol e_h)
       -a_h(\breve{\boldsymbol z},\boldsymbol e_h)\bigr)
 -
 a_h(\breve{\boldsymbol z}-\boldsymbol z_I,\boldsymbol e_h)\right|\\
 &\lesssim h\|\boldsymbol z\|_{2,\gamma} \left(\|\boldsymbol e_h\|_{1,h}
        +h\|\bu\|_{2,\gamma} \right) .
\end{aligned}
\]

Since $\ell(\boldsymbol z;\bu)=a_\gamma(\boldsymbol z,\bu)$, we have 
\[
\begin{aligned}
I_1
 &=\ell(\boldsymbol z;\widearc{\boldsymbol e_h})+ \ell(\boldsymbol z;\bu)-\ell(\boldsymbol z;\widearc{\bu_I})+a_h(\breve{\boldsymbol z}-\boldsymbol z_I,
        \breve{\boldsymbol e})-a_h(\breve{\boldsymbol z},\breve{\bu})+a_h(\breve{\boldsymbol z},\bu_h)\\
 &=\bigl(\ell(\boldsymbol z;\widearc{\boldsymbol e_h})
       -a_h(\breve{\boldsymbol z},\boldsymbol e_h)\bigr)
 +
 \bigl(a_\gamma(\boldsymbol z,\bu)
       -a_h(\breve{\boldsymbol z},\breve\bu)\bigr)\\
       &\qquad -
 \bigl(\ell(\boldsymbol z;\widearc{\bu_I})
       -a_h(\breve{\boldsymbol z},\bu_I)\bigr)
 +
 a_h(\breve{\boldsymbol z}-\boldsymbol z_I,
       \breve{\boldsymbol e}).
\end{aligned}
\]
Consistency error estimates ~\eqref{eq:complete-kernel-surface-consistency},
\eqref{eq:paired-kernel-smooth-consistency}, and
\eqref{eq:paired-kernel-strong-consistency}, followed by the interpolation
estimate in the last term, yield
\[
 |I_1|
 \lesssim h\|\boldsymbol z\|_{2,\gamma}
                \|{\boldsymbol e}_h\|_{1,h}
 +h^2\|\boldsymbol z\|_{2,\gamma}\|\bu\|_{2,\gamma}.
\]

Combining the estimates for the four terms with $
 \|\boldsymbol z\|_{2,\gamma}
 \lesssim\|\boldsymbol e\|_{0,\gamma}$ yields
\[
 \|\boldsymbol e\|_{0,\gamma}
 \lesssim h\|\boldsymbol e_h\|_{1,h}
 +h^2\left(\|\bu\|_{2,\gamma}+\|\bfv\|_{0,\gamma}\right).
\]
Theorem~\ref{thm:arbitrary-topology-H1-error}, the interpolation estimate,
\eqref{eq:surface-piola-norm-equivalence}, and the regularity \eqref{eq:surface-stokes-regularity} prove
\eqref{eq:arbitrary-topology-velocity-L2-error}.
\end{proof}

The two formulations in
Section~\ref{sec:explicit-representative} also satisfy the optimal error
estimates. The proof is given in Appendix~\ref{app:hct-error-estimates}.

\begin{theorem}[Macroelement error estimates]
\label{thm:hct-velocity-errors}
Let $\bu\in\bm H_t^2(\gamma)$ solve \eqref{eq:continuous-stokes}, and let
$\bu_h^M$ and $\bu_h^E$ solve \eqref{eq:hct-direct-method} and
\eqref{eq:hct-virtual-method}, respectively. Under
Assumptions~\ref{ass:mesh} and \ref{ass:physical-data-interface}, for $X\in\{M,E\}$,
\begin{subequations}\label{eq:hct-velocity-errors}
\begin{align}
 \|\breve\bu-\bu_h^X\|_{1,h}
 &\lesssim h\|\bu\|_{2,\gamma}+h^2\|\bfv\|_{0,\gamma},
 \label{eq:hct-H1-error}\\
 \|\breve\bu-\bu_h^X\|_{0,\Gamma_h}
 &\lesssim h^2\|\bfv\|_{0,\gamma}.
 \label{eq:hct-L2-error}
\end{align}
\end{subequations}
\end{theorem}

\section{Numerical experiments}\label{sec:numerical}
We solve the surface Stokes problem~\eqref{eq:continuous-stokes} on
polygonal surface meshes satisfying Assumption~\ref{ass:mesh}, using the
virtual method~\eqref{eq:topology-completed-stokes-method} and, in the
convergence tests, the direct macroelement method~\eqref{eq:hct-direct-method}.
For the four convergence examples, we prescribe the exact velocity, compute
the force from \eqref{eq:continuous-stokes} with zero pressure, and set
$\bfv_h=\bP_h\bfv^e$. All mesh families are refined uniformly with new
vertices placed on the exact surface. The initial mesh is specified for
each example.
We report the energy error and the $L^2$ error of the piecewise affine
$H^1$ projection,
\[
 \begin{aligned}
 E_{a,h}:=a_h^X(I_\Sigma\bu-\bu_h,I_\Sigma\bu-\bu_h)^{1/2},\qquad 
 E_{0,h}:=(\sum_{K\in\Th}
 \|\breve\bu_K-\Pi_1^{\nabla,K}\bu_h\|_{0,K}^2 )^{1/2}.
 \end{aligned}
\]
Here $a_h^X$ denotes the bilinear form of the method under consideration.
For VEM, $X=Z$ and the form is $a_h^Z$ in
\eqref{eq:divergence-free-bilinear-form}; for Macro, $X=M$ and
$a_h^M:=a_h$, the broken Stokes form defined in
Section~\ref{sec:error-estimates}.
The symbols $\bu_h$ and $I_\Sigma$ denote the corresponding numerical
solution and interpolation operator; for Macro they stand for $\bu_h^M$
and $I_\Sigma^M$, respectively.
The projector $\Pi_1^{\nabla,K}$ maps onto affine tangential vector fields on $K$,
preserving the element averages of the velocity and its gradient.
The field $\breve\bu_K$ is the exact Piola pullback on $K$, and
$h=\max_{K\in\Th}\operatorname{diam}(K)$.  For an error $E_h$, the order
between two consecutive meshes of sizes $H$ and $h$ is
$\log(E_H/E_h)/\log(H/h)$.  The tables report the coordinate count
$3\#\Vh+2-\chi(\Gamma_h)$ under Dof.  This count includes one redundant
constant coordinate inherited from the auxiliary scalar space; the number
of independent velocity unknowns is one less.

\paragraph{Necessity of shared edge means.}
We compare the methods based on $\Sigmah$ and $\widetilde\Sigma_h$, defined in
\eqref{eq:global-velocity-space} and
\eqref{eq:uncorrected-global-velocity-space}, respectively.
The space $\widetilde\Sigma_h$ is the surface version of the lowest-order
velocity element, with linear tangential edge traces. The space $\Sigma_h$
is its edge-corrected version, with the same independent velocity degrees
of freedom and zero-mean tangential velocity jumps.
Both methods use the same triangular meshes,
vertex assembly, affine velocity projection, and divergence-preserving
load reconstruction.
On a torus with major radius $R$ and minor radius $r$, set
$\varrho:=R+r\cos\theta$ in toroidal coordinates $(\theta,\varphi)$.
Two harmonic fields are
\begin{equation}\label{eq:torus-harmonic-fields}
 \boldsymbol h_\theta
   :=\frac{1}{\varrho}
      (-\sin\theta\cos\varphi,-\sin\theta\sin\varphi,\cos\theta)^T,\qquad 
 \boldsymbol h_\varphi
   :=\frac{1}{\varrho}(-\sin\varphi,\cos\varphi,0)^T.
\end{equation}
For this comparison, we take $R=2$ and $r=0.7$ and prescribe
$\bu=(R+r\cos\theta)\sin\theta\,\boldsymbol h_\varphi$ and $p=0$.
Table~\ref{tab:svem-edge-correction} shows approximately first-order
convergence in $E_{a,h}$ for both spaces, while $E_{0,h}$ converges at
second order for $\Sigmah$ and first order for $\widetilde\Sigma_h$,
consistent with Remark~\ref{rem:uncorrected-surface-consistency}.
\begin{table}[H]
    \centering
    \caption{Comparison of the velocity spaces on the torus. The asterisks mark suboptimal $L^2$ convergence rates for $\widetilde\Sigma_h$.}
    \label{tab:svem-edge-correction}
    \footnotesize
    \setlength{\tabcolsep}{3pt}
    \begin{tabular*}{\textwidth}{@{\extracolsep{\fill}}lcccccccc@{}}
        \toprule
        & \multicolumn{4}{c}{$\Sigmah$} & \multicolumn{4}{c}{$\widetilde\Sigma_h$} \\
        \cmidrule(lr){2-5}\cmidrule(l){6-9}
        Dof & $E_{a,h}$ & order & $E_{0,h}$ & order
            & $E_{a,h}$ & order & $E_{0,h}$ & order \\
        \midrule
        $384+2$ & 1.171e+01 &      & 2.009e+00 &      & 1.122e+01 &      & 1.529e+00 &      \\
        $1536+2$ & 5.205e+00 & 1.22 & 4.039e-01 & 2.42 & 5.004e+00 & 1.22 & 6.242e-01 & 1.35$^{*}$ \\
        $6144+2$ & 2.482e+00 & 1.08 & 9.034e-02 & 2.19 & 2.386e+00 & 1.08 & 2.957e-01 & 1.09$^{*}$ \\
        $24576+2$ & 1.224e+00 & 1.02 & 2.188e-02 & 2.05 & 1.177e+00 & 1.02 & 1.458e-01 & 1.02$^{*}$ \\
        \bottomrule
    \end{tabular*}
\end{table}

\paragraph{Convergence.}
We consider the following four examples.

\noindent\underline{\textit{Simply connected surface.}}
The first example is posed on the surface
\[
 (x-z^2)^2+y^2+z^2=1.
\]
A parametrization is
$\boldsymbol X(z,\theta)=\bigl(z^2+\sqrt{1-z^2}\cos\theta,
\sqrt{1-z^2}\sin\theta,z\bigr)$.  In the initial mesh, the band
$|z|\le 1/\sqrt2$ is partitioned into quadrilaterals, while the two caps are
triangulated.  The auxiliary scalar field and the prescribed velocity are
$\phi=y$ and $\bu=\curlg y$, respectively.

\noindent\underline{\textit{Torus.}}
The second example uses the torus with $R=1$ and $r=0.6$ and the harmonic
fields in~\eqref{eq:torus-harmonic-fields}.
The initial mesh is induced by a rectangular partition of the
$(\theta,\varphi)$-parameter domain.  With
$\phi(\theta,\varphi):=\sin(3\varphi)\cos(3\theta+\varphi)$, the exact velocity
is $\bu=\curlg\phi+\boldsymbol h_\theta+\boldsymbol h_\varphi$.

\noindent\underline{\textit{Genus-two surface.}}
The third example is obtained from two tori with common major radius
$2$, centres $\boldsymbol c_1=(-2,0,0)$ and $\boldsymbol c_2=(2,0,0)$, axes
$\boldsymbol a_1=(0,0,1)$ and
$\boldsymbol a_2=(0,\sin75^\circ,\cos75^\circ)$, and minor radii
$(R_1^{\mathrm{min}},R_2^{\mathrm{min}})=(1.4,1)$.
For $i=1,2$, let $t_i$ denote the
squared distance to the $i$th centre circle, normalized by
$(R_i^{\mathrm{min}})^2$; thus $t_i=1$ defines the $i$th torus.  The 
surface is 
\[
(3-2t_1)_+^5+(3-2t_2)_+^5=1,
\]
where $a_+:=\max\{a,0\}$.  The surface is of class $C^4$.
If $t_i\ge1.5$ at a point of the
surface, then $t_j=1$ for $j\ne i$, hence the $j$th torus is recovered away from the overlap.

The initial mesh consists of triangles in the overlap region
$t_1,t_2\in(1,1.5)$ and quadrangles on the remaining parts.
With $\bnu$ denoting the outward unit normal, the exact velocity is
\[
 \bu=\bnu\times\left((y,x,0)
       +\frac{(-y,x+2,0)}{2((x+2)^2+y^2)}\right)^{\!T}.
\]
The circle $C:\,(-2+0.6\cos t,0.6\sin t,0)$, $0\le t\le2\pi$,
lies on $\gamma$ and satisfies
$\int_C(\bu\times\bnu)\cdot\mathrm d\bx=\pi$.
This rules out a globally single-valued stream function, since then
$\bu\times\bnu=\gradg\phi$ would have zero circulation.
The velocity therefore has a nonzero harmonic component.

\noindent\underline{\textit{Genus-five surface.}}
The fourth example is posed on
\[
 \gamma=\{(x,y,z):x^4+y^4+z^4-(x^2+y^2+z^2)+0.4=0\}.
\]
This geometry is taken from the FELICITY package
\cite{walker2018felicity,walker2022kirchhoff}.  The initial triangular mesh has
4798 triangles, with local refinement and vertex redistribution to improve
element shapes and reduce the angles between adjacent face normals; all
vertices are projected onto the exact surface.  The exact velocity is
\[
 \bu=\bnu\times
 \left(-\frac{y}{x^2+y^2},\frac{x}{x^2+y^2},0\right)^{\!T}.
\]
For the inner component $C$ of $\gamma\cap\{z=1/\sqrt2\}$, oriented
counterclockwise as viewed from the positive $z$-axis,
$\int_C(\bu\times\bnu)\cdot\mathrm d\bx=2\pi$.
This velocity has no globally single-valued stream function.

Tables~\ref{tab:svem-convergence-general-torus} and
\ref{tab:svem-convergence-higher-genus} compare the two methods on the same
meshes. We observe approximately first- and second-order convergence in the
energy and $L^2$ errors, respectively, for both methods, in agreement with
Theorems~\ref{thm:arbitrary-topology-H1-error},
\ref{thm:arbitrary-topology-lower-order-error}, and
\ref{thm:hct-velocity-errors}.
Figure~\ref{fig:macro-surface-mesh} shows the torus mesh and its HCT refinement.

\begin{figure}[H]
    \centering
    \includegraphics[width=0.32\linewidth]{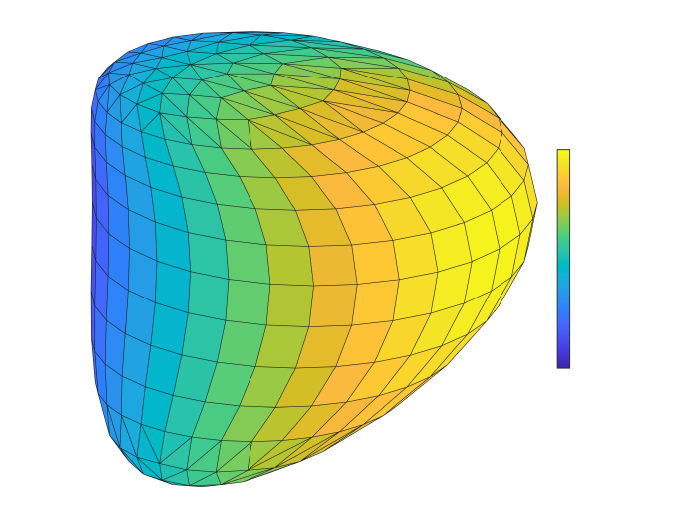}\hfill
    \includegraphics[width=0.32\linewidth]{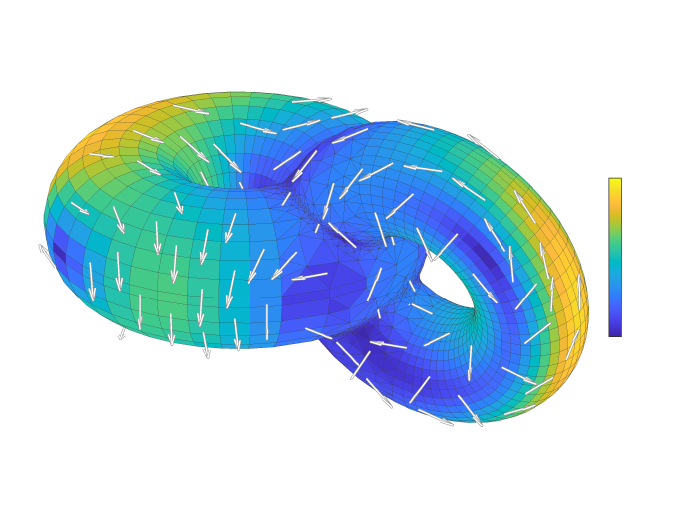}\hfill
    \includegraphics[width=0.32\linewidth]{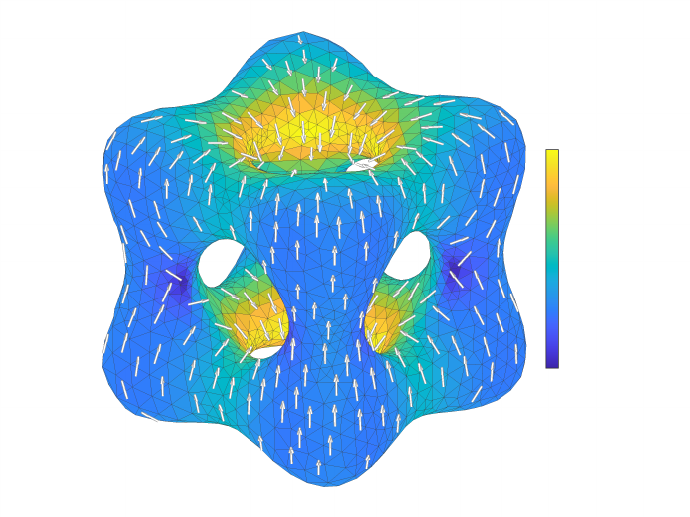}
    \caption{Auxiliary scalar field $\phi_h$ associated with the discrete
    complex on the simply connected surface (left), and projected velocity $\Pi_1^{\nabla,K} \bu_h$ on the genus-two (middle) and genus-five (right)
    surfaces.  In the velocity panels, colors indicate magnitude and arrows
    indicate direction.}
    \label{fig:svem-computed-fields}
    \label{fig:svem-general-scalar}
\end{figure}

\begin{figure}[H]
 \centering
 \begin{subfigure}[t]{.36\linewidth}
  \centering
  \includegraphics[width=\linewidth]{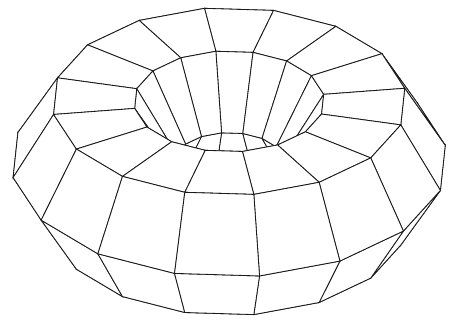}
  \caption{Polygonal mesh}
 \end{subfigure}\hspace{.06\linewidth}
 \begin{subfigure}[t]{.36\linewidth}
  \centering
  \includegraphics[width=\linewidth]{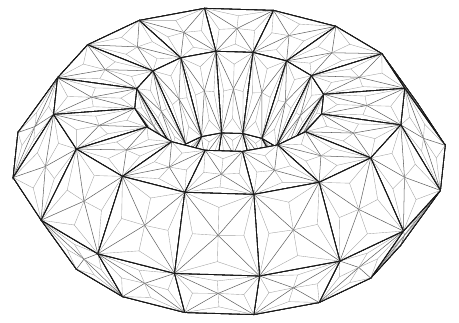}
  \caption{HCT refinement}
 \end{subfigure}
 \caption{Coarsest torus mesh: 96 polygonal faces and 1152 HCT
 subtriangles.}
 \label{fig:macro-surface-mesh}
\end{figure}

\begin{table}[H]
    \centering
    \caption{Numerical errors and convergence rates: simply connected surface and torus.}
    \label{tab:svem-convergence-general-torus}
    \footnotesize
    \setlength{\tabcolsep}{3pt}
    \begin{tabular*}{\textwidth}{@{\extracolsep{\fill}}lcccccccc@{}}
        \toprule
        & \multicolumn{4}{c}{VEM} & \multicolumn{4}{c}{Macro} \\
        \cmidrule(lr){2-5}\cmidrule(l){6-9}
        Dof & $E_{a,h}$ & order & $E_{0,h}$ & order
            & $E_{a,h}$ & order & $E_{0,h}$ & order \\
        \midrule
        \multicolumn{9}{@{}l}{\underline{\textit{Simply connected surface}}}\\[2pt]
        294 & 2.705e+00 &  & 4.301e-01 &  & 2.914e+00 &  & 4.176e-01 &  \\
        1158 & 1.341e+00 & 1.21 & 1.192e-01 & 2.21 & 1.731e+00 & 0.90 & 1.415e-01 & 1.87 \\
        4614 & 6.680e-01 & 1.10 & 3.041e-02 & 2.16 & 8.768e-01 & 1.08 & 3.982e-02 & 2.01 \\
        18438 & 3.406e-01 & 1.02 & 7.976e-03 & 2.02 & 4.581e-01 & 0.98 & 1.104e-02 & 1.94 \\
        \midrule
        \multicolumn{9}{@{}l}{\underline{\textit{Torus}}}\\[2pt]
        $288+2$ & 4.445e+01 &  & 1.075e+01 &  & 2.938e+01 &  & 1.088e+01 &  \\
        $1152+2$ & 3.498e+01 & 0.38 & 3.294e+00 & 1.86 & 3.236e+01 & -0.15 & 4.490e+00 & 1.39 \\
        $4608+2$ & 2.169e+01 & 0.70 & 7.968e-01 & 2.09 & 1.955e+01 & 0.74 & 1.561e+00 & 1.56 \\
        $18432+2$ & 1.136e+01 & 0.94 & 1.999e-01 & 2.01 & 1.031e+01 & 0.93 & 4.383e-01 & 1.84 \\
        \bottomrule
    \end{tabular*}
\end{table}

\begin{table}[H]
    \centering
    \caption{Numerical errors and convergence rates: genus-two and genus-five surfaces.}
    \label{tab:svem-convergence-higher-genus}
    \footnotesize
    \setlength{\tabcolsep}{3pt}
    \begin{tabular*}{\textwidth}{@{\extracolsep{\fill}}lcccccccc@{}}
        \toprule
        & \multicolumn{4}{c}{VEM} & \multicolumn{4}{c}{Macro} \\
        \cmidrule(lr){2-5}\cmidrule(l){6-9}
        Dof & $E_{a,h}$ & order & $E_{0,h}$ & order
            & $E_{a,h}$ & order & $E_{0,h}$ & order \\
        \midrule
        \multicolumn{9}{@{}l}{\underline{\textit{Genus-two surface}}}\\[2pt]
        $8292+4$ & 5.374e+00 &  & 4.418e-01 &  & 5.708e+00 &  & 5.426e-01 &  \\
        $33186+4$ & 2.874e+00 & 0.91 & 1.134e-01 & 1.98 & 3.044e+00 & 0.92 & 1.398e-01 & 1.98 \\
        $132762+4$ & 1.510e+00 & 0.93 & 2.858e-02 & 1.99 & 1.556e+00 & 0.97 & 3.427e-02 & 2.04 \\
        $531066+4$ & 7.766e-01 & 0.96 & 7.258e-03 & 1.98 & 7.914e-01 & 0.98 & 9.080e-03 & 1.92 \\
        \midrule
        \multicolumn{9}{@{}l}{\underline{\textit{Genus-five surface}}}\\[2pt]
        $7173+10$ & 4.027e+00 &  & 2.337e-01 &  & 3.629e+00 &  & 2.578e-01 &  \\
        $28764+10$ & 2.188e+00 & 0.97 & 6.047e-02 & 2.16 & 1.958e+00 & 0.98 & 7.107e-02 & 2.05 \\
        $115128+10$ & 1.148e+00 & 0.96 & 1.602e-02 & 1.97 & 1.021e+00 & 0.97 & 2.115e-02 & 1.80 \\
        $460584+10$ & 5.900e-01 & 0.97 & 4.171e-03 & 1.96 & 5.191e-01 & 0.98 & 6.292e-03 & 1.76 \\
        \bottomrule
    \end{tabular*}
\end{table}

\paragraph{Pressure robustness.}
On the torus and genus-two surface above, we take $q=e^{z/2}$ and prescribe
the pure gradient force $\bfv=\gradg q$, for which the exact velocity is
$\bu=\boldsymbol0$. We use the original VEM with
$\bfv_h=\bP_h(\gradg q)^e$ and refine the same mesh families as in the
convergence tests. Since $I_\Sigma\bu=\boldsymbol0$, the energy error is
$E_{a,h}=a_h^Z(\bu_h,\bu_h)^{1/2}$.
Table~\ref{tab:svem-pressure-robustness} shows second-order decay on both
surfaces, in agreement with the pure gradient estimate in
Theorem~\ref{thm:arbitrary-topology-H1-error}.

\begin{table}[H]
    \centering
    \caption{VEM energy errors for the pure gradient force $\bfv=\gradg e^{z/2}$.}
    \label{tab:svem-pressure-robustness}
    \begin{subtable}[t]{0.45\textwidth}
        \centering
        \small
        \caption{Torus.}
        \begin{tabular*}{\linewidth}{@{\extracolsep{\fill}}lcc@{}}
            \toprule
            Dof & $E_{a,h}$ & order \\
            \midrule
            $288+2$ & 1.50519e-02 &  \\
            $1152+2$ & 3.70651e-03 & 2.21 \\
            $4608+2$ & 9.22143e-04 & 2.05 \\
            \bottomrule
        \end{tabular*}
    \end{subtable}\hfill
    \begin{subtable}[t]{0.45\textwidth}
        \centering
        \small
        \caption{Genus two.}
        \begin{tabular*}{\linewidth}{@{\extracolsep{\fill}}lcc@{}}
            \toprule
            Dof & $E_{a,h}$ & order \\
            \midrule
            $8292+4$ & 1.24331e-02 &  \\
            $33186+4$ & 3.18124e-03 & 1.99 \\
            $132762+4$ & 8.03222e-04 & 1.99 \\
            \bottomrule
        \end{tabular*}
    \end{subtable}
\end{table}

\section{Conclusions and discussion}\label{sec:conclusions}
We have developed pressure-free methods for the surface Stokes problem by extending a planar Stokes complex to polygonal surfaces and introducing a vertex-based edge correction. The resulting complex accounts for the topological components of the full divergence-free space, allowing pressure elimination on closed surfaces of arbitrary genus. The edge correction enforces shared tangential edge means and provides the weak continuity needed for second-order consistency, while preserving the complex and adding no independent degrees of freedom. On the resulting virtual element space, we define a standard VEM formulation and an alternative virtual formulation that is algebraically equivalent to a direct macroelement Galerkin method. Both formulations are stable and achieve optimal convergence rates.

Constructing higher-order VEMs on curved surfaces and domains with curved boundaries raises a further computability issue. On an exact surface or a high-order curved approximation, the metric and tangential projector vary within each element. The local variational integrals therefore involve geometric weights multiplying virtual functions or their derivatives. Mappings for curved-boundary elements lead to analogous weighted quantities. Standard polynomial moments do not generally determine these weighted integrals, so the usual VEM projection argument may no longer suffice to assemble the local forms.

The explicit representatives studied in this paper suggest a way to address this computability difficulty. They are determined by the virtual degrees of freedom and can be evaluated and differentiated at quadrature points. Defining the energy and load through these representatives therefore makes the corresponding geometry-weighted integrals computable by numerical quadrature.
Extending this approach to higher-order VEMs on curved geometries requires analyzing how the choice of representative, the geometric approximation, and quadrature together affect stability and accuracy. This will be the focus of our future work.

\appendix

\section{Proof of Lemma \ref{lem:piola-geometric-comparison}}
\label{app:piola-geometric-estimate}

\begin{proof}[Proof of Lemma~\ref{lem:piola-geometric-comparison}]
Estimate \eqref{eq:pointwise-piola-strain} is
\cite[Lemma~2.2]{demlow2024tangential}, applied to
$\breve{\boldsymbol v}$, and \eqref{eq:local-piola-strain} follows from
\eqref{eq:surface-piola-norm-equivalence}.
Estimates \eqref{eq:piola-form-comparison} and
\eqref{eq:piola-mass-comparison} follow similarly from
\cite[Lemma~4.2 and its proof, (4.8)]{demlow2024tangential}.

For the remaining estimate, let $\boldsymbol v\in\bm H_t^2(\gamma)$ and
let $\bm A\in H^1(\gamma;\mathbb R^{3\times3})$ be symmetric and tangential.
With $\boldsymbol\chi:=-\bnu\times\boldsymbol v$, the cofactor form of
\eqref{eq:surface-piola} is
$\breve{\boldsymbol v}_K=J_K(\bP-d\bH)\boldsymbol\chi^e$.
Differentiating and retaining the term from $\nabla_Kd=\bP_K\bnu$, set
\[
 F(y,\bm m):=
 \bigl[\bm m\times(\nabla\boldsymbol\chi^e\bP_{\bm m})
  -(\bm m\times\bH\boldsymbol\chi)\otimes(\bP_{\bm m}\bnu)\bigr]
  :\bm A.
\]
Here $\bm m\in\mathbb R^3$, all coefficients are evaluated at $y\in\gamma$,
$\bP_{\bm m}=\bI-\bm m\otimes\bm m$, and the matrix cross product acts
columnwise. The chain rule and the geometry bounds give
\[
 \eps_h\breve{\boldsymbol v}:\bm A^e
 =F(\bp(\cdot),\bnu_h)+r_h,\qquad
 \|r_h\|_{L^1(\Gamma_h)}
 \lesssim h^2\|\boldsymbol v\|_{1,\gamma}\|\bm A\|_{0,\gamma}.
\]

Tangentiality gives $F(\cdot,\bnu)=\eps_\gamma(\boldsymbol v):\bm A$.
Polynomial differentiation and the product rule bound the tangential normal
derivative in $W^{1,1}(\gamma)$ and the remaining Taylor coefficients in
$L^1(\gamma)$ by $C\|\boldsymbol v\|_{2,\gamma}\|\bm A\|_{1,\gamma}$.
Applying \cite[Theorem~3.2]{wuZhou2026morleyStokes} and adding $r_h$ gives
\begin{equation}\label{eq:piola-strain-pairing}
 \left|
 (\eps_h\breve{\boldsymbol v},\bm A^e)_{\Gamma_h}
 -(\eps_\gamma(\boldsymbol v),\bm A)_\gamma
 \right|
 \lesssim h^2\|\boldsymbol v\|_{2,\gamma}\|\bm A\|_{1,\gamma}.
\end{equation}
The calculation extends to the stated spaces by density.

Finally, expand both discrete strains around their extended smooth strains.
The mixed terms are controlled by \eqref{eq:piola-strain-pairing} with
$\bm A=\eps_\gamma(\boldsymbol z)$ or $\eps_\gamma(\boldsymbol v)$;
the product of the defects is of order $h^2$ by
\eqref{eq:local-piola-strain}. Together with the area bound and
\eqref{eq:piola-mass-comparison}, this proves
\eqref{eq:paired-kernel-smooth-consistency}.
\end{proof}

\section{Geometric edge estimate}
\label{app:geometric-edge-error}

We prove the geometric edge estimate used in
Lemma~\ref{lem:weak-strain-interpolation}.
For $e=\partial K\cap\partial L$, write
$\jump{\bn}:=\bn_K+\bn_L$ and
$\avg{\bn}:=\tfrac12(\bn_K-\bn_L)$.

\begin{lemma}
\label{lem:geometric-edge-error}
Let $\Gamma_h$ satisfy Assumption~\ref{ass:mesh}, let
$\bm A\in H^1(\gamma;\mathbb R^{3\times3})$ be symmetric and tangential,
and let $\boldsymbol z\in\mathcal Z(\gamma)\cap\bm H_t^2(\gamma)$.
Then
\begin{equation}
 \left|\sum_{e\in\Eh}\int_e
   \jump{\breve{\boldsymbol z}\cdot\bt}\,
   \bt^T\bm A^e\avg{\bn}\,\mathrm{d}s\right|
 \lesssim h^2\|\bm A\|_{1,\gamma}\|\boldsymbol z\|_{2,\gamma}.
 \label{eq:geometric-edge-error}
\end{equation}
\end{lemma}

\begin{proof}
Normal-flux continuity and symmetry of $\bm A$ decompose the negative
of the edge sum in \eqref{eq:geometric-edge-error} as
\[
\begin{aligned}
 &-\sum_{K\in\Th}\int_{\partial K}
       \breve{\boldsymbol z}_K\cdot\bm A^e\bn_K\,\mathrm{d}s+\sum_{e\in\Eh}\int_e\avg{\breve{\boldsymbol z}}
              \cdot\bm A^e\jump{\bn}\,\mathrm{d}s\\
 &+\frac12\sum_{K\in\Th}\int_{\partial K}
       (\breve{\boldsymbol z}_K\cdot\bn_K)
       \bn_K^T\bm A^e\bn_K\,\mathrm{d}s =:J_1+J_2+J_3.
\end{aligned}
\]
For $J_1$, elementwise integration by parts gives
$J_1=-(\eps_h\breve{\boldsymbol z},\bm A^e)_{\Gamma_h}
     -(\breve{\boldsymbol z},\divh\bm A^e)_{\Gamma_h}$.
By \eqref{eq:surface-piola}, the proof of
\cite[Lemma~3.6]{wuZhou2026morleyStokes} applies to the divergence pairing
with the tangential field $-\bnu\times\boldsymbol z$.
Together with \eqref{eq:piola-strain-pairing} and Green's formula on
$\gamma$, this gives
$|J_1|\lesssim h^2\|\bm A\|_{1,\gamma}\|\boldsymbol z\|_{2,\gamma}$.

For $J_2$, \eqref{eq:surface-piola} gives
$|\bm A^e(\breve{\boldsymbol z}_K-\boldsymbol z^e)|
\lesssim h_K^2|\bm A^e||\boldsymbol z^e|$.
The common edge tangent and \eqref{eq:geometry-estimates} give
$|\bP\jump{\bn}|\lesssim h_e^2$.
This bound and the scaled trace inequality thus allow
$\avg{\breve{\boldsymbol z}}$ to be replaced by $\boldsymbol z^e$
in $J_2$ up to an error bounded by
$Ch^2\|\bm A\|_{1,\gamma}\|\boldsymbol z\|_{2,\gamma}$.
The resulting flux is bounded by
\cite[Lemma~3.8]{wuZhou2026morleyStokes}, applied to
$\bm A\boldsymbol z\in[W^{1,1}(\gamma)]^3$, so
$|J_2|\lesssim h^2\|\bm A\|_{1,\gamma}\|\boldsymbol z\|_{2,\gamma}$.

For $J_3$, the orthonormal edge frame $(\bt,\bn_K,\bnu_K)$ gives
$\bn_K^T\bm A^e\bn_K
 =\operatorname{tr}\bm A^e-\bt^T\bm A^e\bt
  -\bnu_K^T\bm A^e\bnu_K$.
The first two terms on the right are common to both faces incident to
$e$, so their contributions to $J_3$ cancel by normal-flux continuity.
Tangentiality of $\bm A$ and $\divK\breve{\boldsymbol z}_K=0$ then yield
\[
\begin{aligned}
 J_3
 &=-\frac12\sum_{K\in\Th}\int_{\partial K}
       (\breve{\boldsymbol z}_K\cdot\bn_K)
       \bnu_K^T\bm A^e\bnu_K\,\mathrm{d}s\\
 &=-\frac12\sum_{K\in\Th}\bigl(\breve{\boldsymbol z}_K,
   \nabla_K[(\bnu_K-\bnu^e)^T\bm A^e(\bnu_K-\bnu^e)]\bigr)_K.
\end{aligned}
\]
The product rule, tangentiality, and
\eqref{eq:geometry-estimates} show that $J_3$ differs from
$(\bnu_h-\bnu^e,(\bm A\bH\boldsymbol z)^e)_{\Gamma_h}$
by at most $Ch^2\|\bm A\|_{1,\gamma}\|\boldsymbol z\|_{2,\gamma}$.
The weak normal estimate \cite[Lemma~3.1]{wuZhou2026morleyStokes},
applied to $\bm A\bH\boldsymbol z\in\bm H_t^1(\gamma)$, therefore gives
$|J_3|\lesssim h^2\|\bm A\|_{1,\gamma}\|\boldsymbol z\|_{2,\gamma}$.

Combining the bounds for $J_1$, $J_2$, and $J_3$ proves
\eqref{eq:geometric-edge-error}.
\end{proof}

\section{Macroelement error estimates}
\label{app:hct-error-estimates}

\begin{proof}[Proof of Lemma~\ref{lem:hct-stability-approximation}]
We first bound the scalar extension in the Hessian seminorm. For
$\phi\in\Phi(K)$, choose $p\in\Pk_1(K)$ by the second-order
Poincar\'e inequality. Prescribe the boundary HCT data of $\phi-p$ and
set all interior HCT data to zero. The resulting function
$s\in\mathcal S_K^{\rm HCT}$ has the same Cauchy trace as $\phi-p$.
Scaling of the HCT basis on the uniformly shape-regular fan, polynomial
inverse inequalities on the boundary edges, and the trace inequality give
\[
 |s|_{2,K}
 \lesssim h_K^{-3/2}\|\phi-p\|_{0,\partial K}
          +h_K^{-1/2}\|\nabla_K(\phi-p)\|_{0,\partial K}
 \lesssim |\phi|_{2,K}.
\]
Since $s+p$ is admissible, the minimizing property gives
$|E_K^\Phi\phi|_{2,K}\lesssim|\phi|_{2,K}$.
Write
\[
 \bv=\tfrac12(\bx-\bx_K^c)q+\curlK\phi,
 \qquad q=\divK\bv\in\Pk_0(K).
\]
Then
$\|q\|_{0,K}\lesssim|\bv|_{1,K}$ and
$|\phi|_{2,K}=|\curlK\phi|_{1,K}\lesssim|\bv|_{1,K}$.
Consequently,
\[
 |E_K^\Sigma\bv|_{1,K}
 \lesssim\|q\|_{0,K}+|E_K^\Phi\phi|_{2,K}
 \lesssim|\bv|_{1,K}.
\]

Set $\boldsymbol\delta=E_K^\Sigma\bv-\bv$. Trace and divergence
preservation give $\boldsymbol\delta\in[H_0^1(K)]^2$ and
$\divK\boldsymbol\delta=0$. Testing the local Stokes equation
in \eqref{eq:local-velocity-space} with $\boldsymbol\delta$ yields
$(\nabla_K\bv,\nabla_K\boldsymbol\delta)_K=0$, hence
\[
 |E_K^\Sigma\bv|_{1,K}^2
 =|\bv|_{1,K}^2+|\boldsymbol\delta|_{1,K}^2.
\]
The zero-trace Poincar\'e inequality therefore gives
\[
 \|\boldsymbol\delta\|_{0,K}
 \lesssim h_K|\boldsymbol\delta|_{1,K}
 \le h_K|E_K^\Sigma\bv|_{1,K}
 \lesssim h_K|\bv|_{1,K}.
\]
For the approximation estimate, set $\bv_I=(I_\Sigma\bv)|_K$ and let
$\boldsymbol p_1$ be the affine approximant in
\eqref{eq:velocity-affine-approximation}. Affine reproduction gives
$(E_K^\Sigma-I)\bv_I=(E_K^\Sigma-I)(\bv_I-\boldsymbol p_1)$.
The $L^2$ bound and the seminorm bounds above, together with
\eqref{eq:general-velocity-interpolation-estimate}, give
\[
\begin{aligned}
 &\|(E_K^\Sigma-I)\bv_I\|_{0,K}
  +h_K|(E_K^\Sigma-I)\bv_I|_{1,K}\\
 &\qquad\lesssim h_K|\bv_I-\boldsymbol p_1|_{1,K}
 \le h_K\bigl(|\bv_I-\breve\bv_K|_{1,K}
              +|\breve\bv_K-\boldsymbol p_1|_{1,K}\bigr)
 \lesssim h_K^2\|\bv\|_{2,\omega_K^\gamma}.
\end{aligned}
\]
The triangle inequality and bounded patch overlap prove
\eqref{eq:hct-interpolation-estimate}. The last assertion follows from
Lemma~\ref{lem:commuting-interpolation} and \eqref{eq:hct-global-commuting}.
\end{proof}

The following lemma is the macroelement analogue of
Lemma~\ref{lem:surface-consistency}.

\begin{lemma}[Geometric consistency for the macroelement space]
\label{lem:hct-surface-consistency}
Under Assumption~\ref{ass:mesh}, let
$\boldsymbol v\in\mathcal Z(\gamma)\cap\bm H_t^2(\gamma)$. Then
\begin{subequations}
\begin{align}
 \lvert
 a_h(\breve{\boldsymbol v},\boldsymbol z_h^M)
 -\ell(\boldsymbol v;\widearc{\boldsymbol z_h^M})
 \rvert
 &\lesssim h\|\boldsymbol v\|_{2,\gamma}
                 \|\boldsymbol z_h^M\|_{1,h}
 &&\forall\boldsymbol z_h^M\in\Sigma_h^M,
 \label{eq:hct-surface-consistency}\\
 \lvert
 a_h(\breve{\boldsymbol v},\boldsymbol z_I^M)
 -\ell(\boldsymbol v;\widearc{\boldsymbol z_I^M})
 \rvert
 &\lesssim h^2\|\boldsymbol v\|_{2,\gamma}
                  \|\boldsymbol z\|_{2,\gamma}
 &&\forall\boldsymbol z\in
       \mathcal Z(\gamma)\cap\bm H_t^2(\gamma),
 \label{eq:hct-paired-surface-consistency}
\end{align}
\end{subequations}
where $\boldsymbol z_I^M:=I_\Sigma^M\boldsymbol z$.
\end{lemma}
\begin{proof}
For $\boldsymbol z_h^M\in\Sigma_h^M$, write
$\boldsymbol z_h^M=E_h^\Sigma\boldsymbol z_h$ with
$\boldsymbol z_h\in\Sigmah$ and set
$\boldsymbol\delta_h=\boldsymbol z_h^M-\boldsymbol z_h$.
The full boundary traces agree on each face, so
$\widearc{\boldsymbol\delta_h}\in\bm H_t^1(\gamma)$ and
$\ell(\boldsymbol v;\widearc{\boldsymbol\delta_h})
=a_\gamma(\boldsymbol v,\widearc{\boldsymbol\delta_h})$.
Thus \eqref{eq:piola-form-comparison} gives
\[
 \lvert a_h(\breve{\boldsymbol v},\boldsymbol\delta_h)
       -\ell(\boldsymbol v;\widearc{\boldsymbol\delta_h})\rvert
 \lesssim h\|\boldsymbol v\|_{1,\gamma}
                  \|\boldsymbol\delta_h\|_{1,h}.
\]
Lemma~\ref{lem:hct-stability-approximation} bounds
$\|\boldsymbol z_h\|_{1,h}+\|\boldsymbol\delta_h\|_{1,h}$ by
$\|\boldsymbol z_h^M\|_{1,h}$. Together with
Lemma~\ref{lem:surface-consistency}, this proves
\eqref{eq:hct-surface-consistency}.
When $\boldsymbol z_h=I_\Sigma\boldsymbol z$, the two interpolation
estimates give $\|\boldsymbol\delta_h\|_{1,h}
\lesssim h\|\boldsymbol z\|_{2,\gamma}$, so the same argument proves
\eqref{eq:hct-paired-surface-consistency}.
\end{proof}

\begin{proof}[Proof of Theorem~\ref{thm:hct-velocity-errors}]
We first consider $\bu_h^M$. Since every
$\boldsymbol z_h^M\in\mathcal Z_h^M$ has a divergence-free Piola lift,
Assumption~\ref{ass:physical-data-interface} gives
\begin{equation}\label{eq:hct-load-consistency}
 \lvert(\bfv_h,\boldsymbol z_h^M)_{\Gamma_h}
       -\ell(\bu;\widearc{\boldsymbol z_h^M})\rvert
 \lesssim h^2\|\bfv\|_{0,\gamma}\|\boldsymbol z_h^M\|_{0,\Gamma_h}.
\end{equation}
Strang's lemma, \eqref{eq:hct-coercivity},
\eqref{eq:hct-interpolation-estimate}, and
Lemma~\ref{lem:hct-surface-consistency} therefore prove
\eqref{eq:hct-H1-error} for $X=M$.
For the $L^2$ estimate, use the duality decomposition in the proof of
Theorem~\ref{thm:arbitrary-topology-lower-order-error}, with
$\bu_h^M$, $I_\Sigma^M\boldsymbol z$, and $I_\Sigma^M\bu$ in place of
$\bu_h$, $\boldsymbol z_I$, and $\bu_I$. The fourth term is zero because
the macroelement method uses $a_h$ itself. The first two terms are bounded
by Lemma~\ref{lem:hct-surface-consistency},
\eqref{eq:paired-kernel-smooth-consistency}, and
\eqref{eq:hct-interpolation-estimate}, and the third by
\eqref{eq:hct-load-consistency}. Hence
\[
 \|\bu-\widearc{\bu_h^M}\|_{0,\gamma}
 \lesssim h\|I_\Sigma^M\bu-\bu_h^M\|_{1,h}
 +h^2\bigl(\|\bu\|_{2,\gamma}+\|\bfv\|_{0,\gamma}\bigr).
\]
The $H^1$ estimate, Piola norm equivalence, and
\eqref{eq:surface-stokes-regularity} prove
\eqref{eq:hct-L2-error} for $X=M$.

For the virtual solution, \eqref{eq:hct-solution-correspondence} and
Lemma~\ref{lem:hct-stability-approximation} give
$\|\bu_h^E-I_\Sigma\bu\|_{1,h}
\lesssim\|\bu_h^M-I_\Sigma^M\bu\|_{1,h}$, which proves
\eqref{eq:hct-H1-error} for $X=E$.
Finally, the two interpolation estimates and
\eqref{eq:hct-local-stability} yield
\[
 \begin{aligned}
 \|\bu_h^M-\bu_h^E\|_{0,\Gamma_h}
 &\le \|(E_h^\Sigma-I)I_\Sigma\bu\|_{0,\Gamma_h}
       +\|(E_h^\Sigma-I)(\bu_h^E-I_\Sigma\bu)\|_{0,\Gamma_h}\\
 &\lesssim h^2\|\bu\|_{2,\gamma}
            +h\|\bu_h^E-I_\Sigma\bu\|_{1,h}
 \lesssim h^2\|\bu\|_{2,\gamma}+h^3\|\bfv\|_{0,\gamma}.
 \end{aligned}
\]
Combining this bound with the macroelement $L^2$ estimate and
\eqref{eq:surface-stokes-regularity} proves
\eqref{eq:hct-L2-error} for $X=E$.
\end{proof}


\begingroup
\small
\begin{thebibliography}{10}

\bibitem{antonietti2014stream}
Paola~F. Antonietti, Louren{\c c}o Beir{\~a}o~da Veiga, David Mora, and Marco
  Verani.
\newblock A stream virtual element formulation of the {Stokes} problem on
  polygonal meshes.
\newblock {\em SIAM Journal on Numerical Analysis}, 52(1):386--404, 2014.

\bibitem{arroyo2009relaxation}
Marino Arroyo and Antonio DeSimone.
\newblock Relaxation dynamics of fluid membranes.
\newblock {\em Physical Review E}, 79(3):031915, 2009.

\bibitem{beirao2019stokescomplex}
Louren{\c c}o Beir{\~a}o~da Veiga, David Mora, and Giuseppe Vacca.
\newblock The {Stokes} complex for virtual elements with application to
  {Navier--Stokes} flows.
\newblock {\em Journal of Scientific Computing}, 81:990--1018, 2019.

\bibitem{bonito2020divergence}
Andrea Bonito, Alan Demlow, and Martin Licht.
\newblock A divergence-conforming finite element method for the surface
  {Stokes} equation.
\newblock {\em SIAM Journal on Numerical Analysis}, 58(5):2764--2798, 2020.

\bibitem{brandner2020error}
Philip Brandner and Arnold Reusken.
\newblock Finite element error analysis of surface {Stokes} equations in stream
  function formulation.
\newblock {\em ESAIM: Mathematical Modelling and Numerical Analysis},
  54(6):2069--2097, 2020.

\bibitem{brueers2026pressure}
Tim Br{\"u}ers, Christoph Lehrenfeld, Tim van Beeck, and Max Wardetzky.
\newblock Releasing the pressure: High-order surface flow discretizations via
  discrete {Helmholtz--Hodge} decompositions.
\newblock arXiv:2603.27714v2, 2026.
\newblock DOI: \url{https://doi.org/10.48550/arXiv.2603.27714}.

\bibitem{brueers2025streamfunction}
Tim Br{\"u}ers, Christoph Lehrenfeld, and Max Wardetzky.
\newblock Streamfunction-vorticity formulation for incompressible viscid and
  inviscid flows on general surfaces.
\newblock arXiv:2512.20763, 2025.

\bibitem{chen2024extrinsic}
Chunyu Chen, Xuehai Huang, and Huayi Wei.
\newblock Virtual element methods without extrinsic stabilization.
\newblock {\em SIAM Journal on Numerical Analysis}, 62(1):567--591, 2024.

\bibitem{chinosi2016virtual}
Claudia Chinosi and L.~Donatella Marini.
\newblock Virtual element method for fourth order problems: {$L^2$}-estimates.
\newblock {\em Computers \& Mathematics with Applications}, 72(8):1959--1967, 2016.
\newblock DOI: \url{https://doi.org/10.1016/j.camwa.2016.02.001}.

\bibitem{demlow2024tangential}
Alan Demlow and Michael Neilan.
\newblock A tangential and penalty-free finite element method for the surface
  {Stokes} problem.
\newblock {\em SIAM Journal on Numerical Analysis}, 62(1):248--272, 2024.

\bibitem{demlow2026taylorhood}
Alan Demlow and Michael Neilan.
\newblock A {Taylor--Hood} finite element method for the surface {Stokes}
  problem without penalization.
\newblock {\em SIAM Journal on Numerical Analysis}, 64(2):565--600, 2026.

\bibitem{frerichs2022pressure}
Derk Frerichs and Christian Merdon.
\newblock Divergence-preserving reconstructions on polygons and a really
  pressure-robust virtual element method for the {Stokes} problem.
\newblock {\em IMA Journal of Numerical Analysis}, 42(1):597--619, 2022.

\bibitem{frittelli2018surfacevem}
Massimo Frittelli and Ivonne Sgura.
\newblock Virtual element method for the {Laplace--Beltrami} equation on
  surfaces.
\newblock {\em ESAIM: Mathematical Modelling and Numerical Analysis},
  52(3):965--993, 2018.

\bibitem{hardering2025parametric}
Hanne Hardering and Simon Praetorius.
\newblock Parametric finite-element discretization of the surface {Stokes}
  equations: Inf-sup stability and discretization error analysis.
\newblock {\em IMA Journal of Numerical Analysis}, 45(5):2948--2987, 2025.

\bibitem{henle2010hydrodynamics}
Mark~L. Henle and Alex~J. Levine.
\newblock Hydrodynamics in curved membranes: The effect of geometry on
  particulate mobility.
\newblock {\em Physical Review E}, 81(1):011905, 2010.

\bibitem{holst2012geometric}
Michael Holst and Ari Stern.
\newblock Geometric variational crimes: {Hilbert} complexes, finite element
  exterior calculus, and problems on hypersurfaces.
\newblock {\em Foundations of Computational Mathematics}, 12(3):263--293, 2012.

\bibitem{jankuhn2018incompressible}
Thomas Jankuhn, Maxim~A. Olshanskii, and Arnold Reusken.
\newblock Incompressible fluid problems on embedded surfaces: Modeling and
  variational formulations.
\newblock {\em Interfaces and Free Boundaries}, 20(3):353--377, 2018.

\bibitem{john2017divergence}
Volker John, Alexander Linke, Christian Merdon, Michael Neilan, and Leo~G.
  Rebholz.
\newblock On the divergence constraint in mixed finite element methods for
  incompressible flows.
\newblock {\em SIAM Review}, 59(3):492--544, 2017.

\bibitem{kone2026divergence}
Yerim Kone, Michael Neilan, and David Poling.
\newblock A divergence-free {Scott--Vogelius} finite element method for the
  surface {Stokes} problem.
\newblock arXiv:2606.07840, 2026.

\bibitem{meng2023stokesstability}
Jian Meng, Louren{\c c}o Beir{\~a}o~da Veiga, and Lorenzo Mascotto.
\newblock Stability and interpolation properties for {Stokes}-like virtual
  element spaces.
\newblock {\em Journal of Scientific Computing}, 94(3):56, 2023.

\bibitem{neilan2025c0}
Michael Neilan and Hongzhi Wan.
\newblock A {$C^0$} interior penalty method for the stream function formulation
  of the surface {Stokes} problem.
\newblock {\em ESAIM: Mathematical Modelling and Numerical Analysis},
  59(2):1177--1211, 2025.

\bibitem{olshanskii2021infsup}
Maxim~A. Olshanskii, Arnold Reusken, and Alexander Zhiliakov.
\newblock Inf-sup stability of the trace {$\mathbf{P}_2$--$P_1$} {Taylor--Hood}
  elements for surface {PDEs}.
\newblock {\em Mathematics of Computation}, 90(330):1527--1555, 2021.

\bibitem{percell1976clough}
Peter Percell.
\newblock On cubic and quartic {Clough--Tocher} finite elements.
\newblock {\em SIAM Journal on Numerical Analysis}, 13(1):100--103, 1976.

\bibitem{reusken2020stream}
Arnold Reusken.
\newblock Stream function formulation of surface {Stokes} equations.
\newblock {\em IMA Journal of Numerical Analysis}, 40(1):109--139, 2020.

\bibitem{walker2018felicity}
Shawn~W. Walker.
\newblock {FELICITY}: A {Matlab/C++} toolbox for developing finite element
  methods and simulation modeling.
\newblock {\em SIAM Journal on Scientific Computing}, 40(2):C234--C257, 2018.

\bibitem{walker2022kirchhoff}
Shawn~W. Walker.
\newblock The {Kirchhoff} plate equation on surfaces: The surface
  {Hellan--Herrmann--Johnson} method.
\newblock {\em IMA Journal of Numerical Analysis}, 42(4):3094--3134, 2022.

\bibitem{wu2025biharmonic}
Shuonan Wu and Hao Zhou.
\newblock A stabilized nonconforming finite element method for the surface
  biharmonic problem.
\newblock {\em SIAM Journal on Numerical Analysis}, 63(4):1642--1665, 2025.

\bibitem{wuZhou2026morleyStokes}
Shuonan Wu and Hao Zhou.
\newblock Stabilized {Morley} {FEM} for surface {Stokes} in stream-function
  formulation: Optimal convergence via a new geometric estimate, 2026.
\newblock arXiv:2607.26664.

\end{thebibliography}
\endgroup
\end{document}